\documentclass{amsart}[11pt]

\usepackage{etex}
\usepackage[lmargin=1in,rmargin=1in,tmargin=1in,bmargin=1in]{geometry}
\usepackage{amsmath,amsthm,amsfonts,amssymb,verbatim}
\usepackage{graphicx}
\usepackage{tabularx}
\usepackage{overpic}
\usepackage{tikz}\usetikzlibrary{decorations.markings}
\usetikzlibrary{arrows}
\usepackage{rotating}
\usepackage{hyperref}
\usepackage{framed}
\usepackage{xcolor}
\usepackage{placeins}
\usepackage{booktabs}
\usepackage{pinlabel}
 \usepackage{picins}
\usepackage{color}
\usepackage{multirow}
\usepackage{overpic}
\usepackage{colortbl}
\usepackage{geometry}
\usepackage{slashed}
\usepackage[font=footnotesize,labelfont=bf]{caption}
\usepackage{arydshln}
\usepackage{blkarray}

\usepackage{wrapfig}

\DeclareMathAlphabet{\mathpzc}{OT1}{pzc}{m}{it}

\definecolor{darkBlue}{RGB}{85,123,151}
\definecolor{lightBlue}{RGB}{114,178,223}
\definecolor{lightPeach}{RGB}{245,178,157}
\definecolor{darkPeach}{RGB}{226,103,67}
\definecolor{darkPeach}{RGB}{226,103,67}
\definecolor{lightGreen}{RGB}{82,157,54}
\definecolor{darkGreen}{RGB}{20,105,0}
\definecolor{lightBrown}{RGB}{190,166,159}

\newcommand{\me}{\textsc{m}}

\def\co{\colon\thinspace}

\newcommand{\floor}[1]{\left\lfloor #1 \right\rfloor}
\newcommand{\ceiling}[1]{\left\lceil #1 \right\rceil}

\newcommand{\pt}{\text{pt}}

\newcommand{\Z}{\mathbb{Z}}

\newcommand{\R}{\mathbb{R}}
\newcommand{\F}{\mathbb{F}}

\newcommand{\sI}{\mathcal{I}}

\newcommand{\sP}{\mathcal{P}}

\newcommand{\cylinder}{\overline{T}}
\newcommand{\pcylinder}{\cylinder_\bullet}
\newcommand{\torus}{T}
\newcommand{\ptorus}{\torus_\bullet}
\newcommand{\plane}{\widetilde{T}}
\newcommand{\pplane}{\plane_\bullet}
\newcommand{\curveset}{\boldsymbol{\gamma}}

\newcommand{\HFhat}{\widehat{\mathit{HF}}}

\newcommand{\CFKminus}{\mathit{CFK}^-}

\newcommand{\HFK}{\widehat{\mathit{HFK}}}

\newcommand{\Alg}{\mathcal{A}}

\newtheorem{theorem}{Theorem}%[section]

\newtheorem{corollary}[theorem]{Corollary}
\newtheorem{proposition}[theorem]{Proposition}

\newtheorem*{namedtheorem}{\theoremname}
\newcommand{\theoremname}{testing}

\theoremstyle{definition}

\newtheorem{remark}[theorem]{Remark}

\title[Cabling in terms of immersed curves]{Cabling in terms of immersed curves}
\date{\today}
\author[Jonathan Hanselman]{Jonathan Hanselman}
\address {Department of Mathematics, Princeton University.\newline \it{E-mail address:} \tt{jh66@math.princeton.edu}}

\author[Liam Watson]{Liam Watson}
\thanks{JH was partially supported by NSF grant DMS-1812527; LW was partially supported by an NSERC discovery/accelerator grant}
\address {Department of Mathematics, University of British Columbia.\newline \it{E-mail address:} \tt{liam@math.ubc.ca}}

\begin{document}
\maketitle

\begin{abstract} 
In joint work with J. Rasmussen, we gave an interpretation of Heegaard Floer homology for manifolds with torus boundary in terms of immersed curves in a punctured torus \cite{HRW}. In particular, knot Floer homology is captured by this invariant; see \cite{HRW-companion}. Appealing to earlier work of the authors on bordered Floer homology \cite{HW}, we give a formula for the behaviour of these immersed curves under cabling. 
\end{abstract}

%\section{Introduction}\label{sec:intro}
% !TEX root = ../cabling.tex

% \section{Introduction}
% \label{sec:intro}

Knot Floer homology, as introduced by Ozsv\'ath and Szab\'o \cite{OSz2004-knot} and Rasmussen \cite{Rasmussen-thesis}, provides a categorification: Given a knot $K$ in the three-sphere this invariant is a bigraded vector space $\bigoplus_{a,m\in\Z}\HFK_m(K,a)$ with the property that $\sum_{a,m}(-1)^m\dim\big(\HFK_m(K,a)\big)t^a$ recovers the (symmetrized) Alexander polynomial. This polynomial knot invariant satisfies natural properties associated with operations on knots, for instance, it is well behaved under cabling. Understanding how this particular property manifests at the categorified level drove some of the early calculations of knot Floer homology; see in particular work of Hedden \cite{Hedden-thesis,Hedden2005, Hedden2009}. 

\labellist \tiny
\pinlabel $(2,-1)$ at 19 8, \pinlabel $(2,1)$ at 124 8
\endlabellist
\parpic[r]{
 \begin{minipage}{55mm}
 \centering
 \includegraphics[scale=.75]{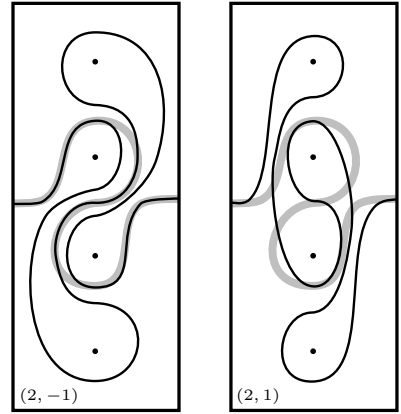}
 \captionof{figure}{The Heegaard Floer homology for the $(2,-1)$ and $(2,1)$ cables  of the right-hand trefoil; the invariant for the trefoil complement is shown in grey.}
 \label{fig:intro-cables}
  \end{minipage}%
  }
Bordered Floer homology provides an essential tool for studying decompositions of three-manifolds along essential tori \cite{LOT}. A framework of bimodules, of relevance to satellite operations, is laid out in \cite{LOT-bimodules}. The work of Levine \cite{Levine2012}, Hom \cite{Hom2014}, and Petkova \cite{Petkova2013}, for example, puts this to use in an essential way. In the setting of manifolds with a single toroidal boundary component, the relevant bordered invariants have been recast in terms of immersed curves in the once-punctured torus \cite{HRW,HRW-companion}. For the purpose of this note, the examples of interest will be provided by the complement of a knot in the three-sphere: Our aim is to establish formulae for how these invariants behave under cabling. Namely, for a knot $K$ in $S^3$ let $K_{p,q}$ denote the $(p,q)$-cable of $K$, and denote the respective knot complements by $M = S^3\setminus \nu(K)$ and 
$M_{p,q} = S^3\setminus \nu(K_{p,q})$; given the immersed multicurve $\HFhat(M)$ we wish to describe $\HFhat(M_{p,q})$ explicitly. For example, immersed curves for two cables of the right-hand trefoil are illustrated in Figure \ref{fig:intro-cables}; the expert reader already familiar with the passage from $\HFK(K)$ to $\HFK(K_{p,q})$ should compare these pictures with the detailed calculations in \cite{Hedden-thesis} or in \cite{OSS2017}. Our calculation makes an explicit appeal to a bordered trimodule calculated by the first author \cite{Hanselman2016}, which was reinterpreted combinatorially in work of the authors predating the immersed curves invariant \cite{HW}. Indeed, central to this article is the work of  translating our merge operation (described in terms of loop calculus) into the language of immersed curves (Section \ref{sec:curves}); cabling is then seen as a special case of the merge operation (Section \ref{sec:proof}).

Recall from \cite{HRW, HRW-companion} that for a (connected, orientable) three-manifold $M$ with torus boundary, the invariant $\HFhat(M)$ takes the form of a collection of immersed curves, possibly decorated with local systems, in the punctured torus $T_\bullet = \partial M \setminus z$, where $z$ is some fixed base point in $\partial M$. If we choose a pair of parametrizing curves $(\alpha, \beta)$ on $\partial M$ then $T_\bullet$ can be identified with the square $[0,1]\times[0,1]$ with opposite sides identified such that $\alpha$ runs in the positive vertical direction, $\beta$ runs in the positive horizontal direction, and the puncture $z$ identified with $(0,0)$. For a knot complement, there is a preferred choice of parametrizing curves is $(\mu, \lambda)$ where $\mu$ is the meridian and $\lambda$ is the Seifert longitude. The invariant $\HFhat(M)$ comes equipped with grading data which, among other things, specifies a lift of these curves to the punctured cylinder $\pcylinder = (\R^2\setminus \Z^2) / \langle \lambda \rangle$; in the standard framing, this can be identified with $(\R/\Z)\times \R$ with punctures at each lattice point $(0,n)$. (Note that this is the point of view taken in the presentation of the invariants in Figure \ref{fig:intro-cables}: in each rectangle, the sides are identified to form a cylinder.) Thus as a graded object it makes sense to view $\HFhat(M)$ as a collection of closed immersed curves $\curveset = (\gamma_0, \ldots, \gamma_n)$ in $\pcylinder$, defined up to homotopy in $\pcylinder$, possibly decorated with local systems. For knots in $S^3$, these curves have the property that, possibly after a homotopy, the curve set intersects the vertical line $\{1/2\}\times \R$ exactly once; we will always assume that $\gamma_0$ is the curve component containing this intersection. In other words, $\gamma_0$ wraps around the cylinder exactly once, while the remaining $\gamma_{i>0}$ can be confined to a neighborhood of the vertical line through the punctures. We remark also that, while $\curveset$ may carry non-trivial local systems, $\gamma_0$ always carries the trivial one-dimensional local system (otherwise the rank of $\HFhat$ of the meridional filling of $K$ would be greater than one). Finally, it is sometimes convenient to work in the plane $\pplane = \R^2 \setminus \Z^2$ rather than the cylinder, with the multicurve $\curveset$ lifting to one that is invariant under translation by $\lambda$; note that in this cover $\gamma_0$ lifts to a single periodic curve while each $\gamma_{i>0}$ lifts to infinitely many copies of the same curve.

We will show that the $(p,q)$-cable operation acts on $\HFhat(M)$ by applying a particular diffeomorphism to the plane. Let $\boldsymbol g_{p,q}$ be a diffeomorphism of $\R^2$ defined on the lattice $\Z^2$ by sliding each lattice point leftward along lines of slope $\tfrac q p$ until they first meet a vertical line $x = np$ for some integer $n$. Note that $\boldsymbol g_{p,q}$ does not fix the lattice $\Z^2$ but rather takes it to $p\Z \times \tfrac 1 p\Z$; let $\boldsymbol f_{p,q}$ be the composition of this map with vertical stretching by a factor of $p$ and horizontal compression by a factor of $p$, so that $\boldsymbol f_{p,q}$ takes $\Z^2$ to $\Z^2$, followed by a vertical shift of $\tfrac{ (p-1)(q-1) }{2}$. We remark that the vertical shift is forced by the symmetry of the curves $\HFhat(M)$ for any $M$ and our convention that these curves are centered at height $\tfrac 1 2$; with this convention understood we will generally ignore the vertical positioning of the curves, but it is sometimes helpful to keep track of this vertical translation explicitly. The map $\boldsymbol f_{p,q}$ is not linear, though in some sense it is as close to being linear as possible: it is the composition of linear transformations, which can each be realized as a sequence of plane shears, with a single \emph{fractional plane shear} (defined in Section \ref{sec:proof}). 
\begin{theorem}\label{thm:main-cable}
If $\curveset$ is the immersed multicurve associated with $K$, $\curveset_{p,q}$ is the immersed multicurve associated with $K_{p,q}$, and $\tilde\curveset$ and $\tilde\curveset_{p,q}$ are the corresonding lifts to $\pplane = \R^2 \setminus \Z^2$, then $\tilde\curveset_{p,q}$ is  homotopic to $\boldsymbol f_{p,q}(\tilde\curveset)$.
\end{theorem}
Note that $\boldsymbol f_{p,q}$ is periodic with period $p$ in the horizontal direction, so it makes sense to view $\boldsymbol f_{p,q}$ as a map from the cylinder $p\cylinder := (\R/p\Z)\times \R$ to $\cylinder := (\R/\Z)\times \R$ taking lattice points to lattice points. With this view, the process of computing $\HFhat(M_{p,q})$ from $\HFhat(M)$ is to lift from $\cylinder$ to $p\cylinder$ and then apply $\boldsymbol f_{p,q}$. In practice, this amounts to drawing $p$ copies of $\HFhat(M)$ in sequence, perturbing the curve by pushing lattice points along lines of slope $\tfrac q p$ until they all lie on the same vertical line, and then scaling vertically by a factor of $p$. This procedure is depicted in Figure \ref{fig:intro-example} for the case of the $(3,2)$-cable of the right-handed trefoil. It is helpful to note that the procedure of pushing lattice points along lines of slope $\tfrac q p$ can equivalently be viewed as drawing $p$ copies of the input curve with staggered heights and then translating punctures horizontally. In practice, the process for computing $\HFhat(M_{p,q})$ from $\HFhat(M)$ amounts to a three step process:
\begin{itemize}
\item[(1)] Draw $p$ copies of $\HFhat(M)$ next to each other, each scaled vertically by a factor of $p$, staggered in height such that each copy of the curve is a height of $q$ units lower than the previous copy;
\item[(2)] Connect the loose ends of the successive copies of the curve; and
\item[(3)] Translate the pegs horizontally so that they lie in the same vertical line, carrying the curve along with them.
%\item[(4)] Shift vertically so that the new curve is centered at height $\tfrac 1 2$
\end{itemize}

%Strictly speaking, the procedure above only defines $\HFhat(M_{p,q})$ up to an overall vertical shift, but the vertical position can be normalized using the convention that the curve is centered on the line $y = \tfrac 1 2$. To compute the needed vertical shift explicitly, note that a given lattice point $(i,n)$ in $(\Z/p\Z)\times \Z$ is taken by $\boldsymbol f_{p,q}$ to $(0, pn - iq)$. Assuming that $\HFhat(M)$ is centered at height $\tfrac 1 2$, the first copy of $\HFhat(M)$ in the above construction, after scaling by a factor of $p$, is centered at height $\tfrac p 2$ and the last copy is centered at height $\tfrac p 2 - (p-1)q$. It is clear that the overall curve is centered at the average of these two heights, and this center must be shifted to height $\tfrac 1 2$. Thus $\boldsymbol f_{p,q}$ should be followed with a vertical shift of
%$$\textstyle\frac 1 2 - \left(\frac p 2 - \frac{(p-1)q}{2}\right) = \frac{(p-1)(q-1)}{2}$$

\begin{figure}
\labellist 

\endlabellist
\includegraphics[scale=1.2]{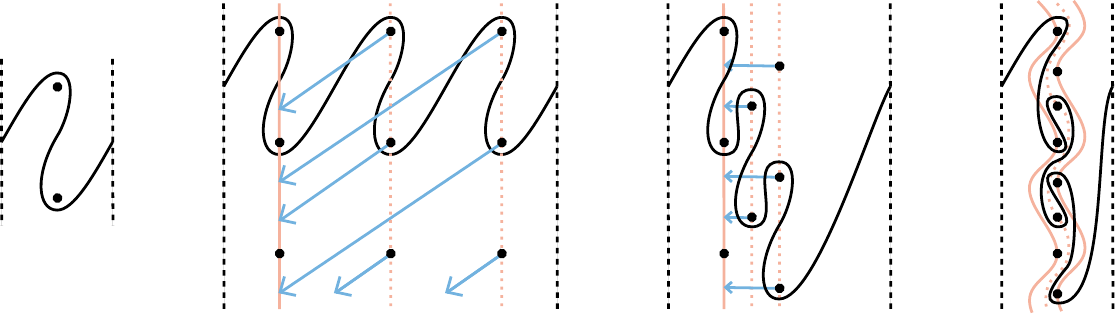}
\caption{Computation of the immersed curve associated with the $(3,2)$-cable of the right handed trefoil, starting from the trefoil curve pictured on the left. The two middle diagrams are two ways of thinking about the construction starting from three copies of the trefoil curve: we either slide lattice points along lines of slope $\tfrac 2 3$ or we stagger the heights of the three copies of the trefoil curve and then slide lattice points horizontally. Either way the result is the curve on the right.}\label{fig:intro-example}
\end{figure}

\subsection*{Numerical concordance invariants extracted from curves.} As an illustration of Theorem \ref{thm:main-cable} at work, we can revisit the work of Hedden \cite{Hedden2005,Hedden2009} and Van Cott \cite{VanCott2010}, culminating in a result of Hom \cite{Hom2014}, which establishes the behaviour of the $\tau$-invariant under cabling. Since $\tau(K)$ can be easily extracted from the immersed multicurve $\HFhat(M)$, we can recover this cabling behavior from Theorem \ref{thm:main-cable}. The same is true for some other numerical invariants. We begin by making an observation that is implicit in earlier work: Let $\curveset = (\gamma_0, \ldots, \gamma_n)$ denote the underlying set of immersed curves for $\HFhat(M)$, with $\gamma_0$ the unique component which wraps around the cylinder. This component is itself an invariant of $K$, so it will sometimes be convenient to express it as $\gamma_0(K)$.

\begin{proposition}The curve $\gamma_0(K)$ is an invariant of the concordance class of $K$.\end{proposition}

\begin{proof} This follows from \cite{Hom-survey} and the recipe for deriving $\HFhat(M)$ from $\CFKminus(K)$ described in \cite[Section 4]{HRW-companion}. The concordance invariant described in \cite{Hom-survey} is the smallest direct summand of $\CFKminus(K)$, up to homotopy equivalence, which supports the homology of $S^3$. The set of immersed curves derived from this summand is a subset of the immersed curves $\HFhat(M)$ which necessarily contains $\gamma_0(K)$. This subset of curves, and in particular $\gamma_0(K)$, is thus a concordance invariant. \end{proof}

Note that the concordance invariant described in \cite{Hom-survey} is slightly stronger than $\gamma_0$ since some information may be lost when passing from complexes to immersed curves (namely, diagonal arrows are ignored). In fact, $\gamma_0(K)$ carries exactly the same information as the $\epsilon$-equivalence class of $K$ defined in \cite{Hom-survey}. Any number that can be extracted from $\gamma_0$ is automatically a concordance invariant, and several familiar concordance invariants can be defined in this way. The two most common are $\tau$ and Hom's $\epsilon$-invariant, which are extracted from $\gamma_0$ as follows: Starting on the section of $\gamma_0$ which wraps around the back of the cylinder (say, at the unique intersection of $\gamma_0$ with the line $x = \tfrac 1 2$), and moving rightward along $\gamma_0$, let $a$ denote the first intersection of $\gamma_0$ with the vertical axis $x = 0$. Then the integer $\tau$ records the height of the intersection point $a$ (here we use a discrete notion of height given by the greatest integer less than the $y$-coordinate of $a$). Continuing along $\gamma_0$ from $a$, one of three things can happen: $\gamma_0$ can turn downwards, it can turn upwards, or it can continue straight to wrap around the cylinder. This is recorded by $\epsilon$, which takes the values $+1$, $-1$, or $0$ in these three cases, respectively. (Both of these observations are made in \cite{HRW-companion}.) Note that if $\epsilon = 0$ then there is only one intersection of $\gamma_0$ with the vertical axis, so $\gamma_0$ is simply a horizontal curve, which is the immersed curve associated with the complement of the unknot. Now consider the effect of cabling on each of these invariants. Throughout, let $\gamma_0 = \gamma_0(K)$ and let $\gamma'_0 = \gamma_0(K_{p,q})$.

\begin{theorem}[{Hom \cite[Theorem 2]{Hom2014}}]\label{thm:epsilon}
If $\epsilon(K)=\pm1$ then $\epsilon(K_{p,q})=\epsilon(K)$; and if $\epsilon(K)=0$ then $\epsilon(K_{p,q})=\epsilon(T_{p,q})$.
\end{theorem}
%First recall that $\epsilon(K)$, which takes one of only 3 possible values, can be interpreted as recording how the curve-set $\HFhat(M)$ behaves immediately after its first intersection with the vertical line through the punctures after wrapping around the cylinder and approaching this line from the left (compare \cite[Section 4]{HRW-companion}). In particular, $\epsilon(K)$ is $1, -1$, or $0$ depending on whether the curve turns downward, turns upward, or continues straight to the right after this intersection. Note that this depends only on the unique homologically nontrivial component of $\HFhat(M)$, i.e. the component which wraps around the cylinder, which we denote $\gamma_0$.
\begin{proof}[A quick re-proof of Theorem \ref{thm:epsilon}]
By Theorem \ref{thm:main-cable}, $\gamma'_0$ is obtained from $\gamma_0$ by placing $p$ copies of $\gamma_0$ next to each other, with appropriate vertical shifts, and compressing them into one vertical line. The first intersection of $\gamma'_0$ with the vertical axis thus comes from the first intersection of the first copy of $\gamma_0$ with the vertical axis, and clearly if $\gamma_0$ turns upward or downward at this point than $\gamma'_0$ does also. On the other hand if $\epsilon(K) = 0$ then $\gamma_0$ is simply a horizontal line, the same as the curve associated with the unknot. It follows that $\gamma'_0$ agrees with $\gamma_0(T_{p,q})$, since $T_{p,q}$ is the $(p,q)$-cable of the unknot, and thus $\epsilon(K_{p,q}) = \epsilon(T_{p,q})$.
 %If $\epsilon(K) = 0$ then $\HFhat(M)$ is horizontal line at some height between 0 and 1. If $|q| = 1$, then each of the $p$ copies of $\HFhat(M)$ are shifted vertically by less than 1; thus we can take the horizontal line to be at height $\epsilon$ if $q = 1$ or $1-\epsilon$ if $q = -1$ through all $p$ copies of $\HFhat(M)$. This curve remains horizontal when the $p$ columns are compressed into one, so $\epsilon(K_{p,q}) = 0. On the other hand, if $q > 1$, then the $\floor{\tfrac p q}$th copy of the curve is shifted downward by greater than 1. After passing over the height 0 lattice point from the first copy of $\HFhat(M)$, the curve must bend downward in order to pass under (the image of) the height one lattice point from this copy, and so $\epsilon(K) = -1$. A mirror argument shows that if $q < -1$ then $\epsilon(K) = 1$
% this first intersection $\HFhat(M_{p,q})$ continues to (the image under $\psi'_{p,q}$ of ) the second copy of $\HFhat(M)$. Since the second copy of is shifted downward by $q/p$, $\HFhat(M_{p,q})$ turns downward if $q > 0$ and upward if $q < 0$.
\end{proof}

The value $\epsilon(T_{p,q})$ was also computed in \cite[Theorem 2]{Hom2014}; we can recover this computation by viewing $T_{p,q}$ as the $(p,q)$-cable of the unknot. In this case $\gamma_0$ is horizontal in $\cylinder$ and lifts to a horizontal line in $p\cylinder$. To compute $\gamma'_0$ from this we shift the $i$th column downwards by $\frac{iq}{p}$ and then compress horizontally. If $|q| = 1$ then $T_{p,q}$ is unknotted and we must have $\epsilon(T_{p,q}) = 0$; indeed, in this case every column of lattice points shifts by less than one unit, so it is possible for $\gamma'_0$ to remain horizontal despite the shift. On the other hand, if $q >1$ then the shift causes $\gamma'_0$ to turn downwards, so $\epsilon(T_{p,q}) = +1$; similarly, if $q < -1$ then $\gamma'_0$ turns upward and $\epsilon(T_{p,q}) = -1$.

 \begin{theorem}[{Hom \cite[Theorem 1]{Hom2014}}]\label{thm:tau}
If $\epsilon(K)=\pm1$ then $\tau(K_{p,q})=p\tau(K)+\frac{(p-1)(q\mp1)}{2}$; and if $\epsilon(K)=0$ then $\tau(K_{p,q})=\tau(T_{p,q}) = (-1)^{\operatorname{sign}(q)} \tfrac{(p-1)(|q|-1)}{2}$.%, where $\pm$ is determined by the sign of $q$. 
\end{theorem}

\labellist \tiny
\pinlabel $\tau$ at 27 17
\pinlabel $\tau+1$ at 35 65
\pinlabel $p\tau+\tfrac{(p-1)(q-1)}{2}$ at 127 17
\pinlabel $p\tau+p+\tfrac{(p-1)(q-1)}{2}$ at 135 65
\endlabellist
\parpic[r]{
 \begin{minipage}{60mm}
 \centering
 \includegraphics[scale=1]{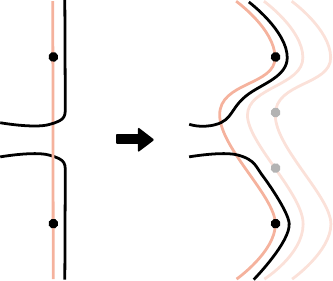} \hspace{2.5 cm}
 \captionof{figure}{Calculating $\tau(K_{p,q})$.}
 \label{fig:tau-proof}
  \end{minipage}%
  }
\begin{proof}[A quick re-proof of Theorem \ref{thm:tau}]
The first intersection of $\gamma'_0$ with the vertical axis clearly comes from the first intersection of the first copy of $\gamma_0$ with the vertical axis. This intersection occurs between the lattice points at height $\tau(K)$ and $\tau(K)+1$; after applying $\boldsymbol f_{p,q}$ and the appropriate vertical shift, these lattice points map to heights $h_1 = p\tau(K) + \tfrac{(p-1)(q-1)}{2}$ and $h_2 = p\tau(K) + p + \tfrac{(p-1)(q-1)}{2}$.  Note that there are $p-1$ lattice points between these two heights; whether or not $\gamma'_0$ first intersects the vertical axis above or below these points depends on the behavior of $\gamma_0$ just after it crosses the vertical axis, as pictured in Figure \ref{fig:tau-proof}. If $\gamma_0$ turns downward (i.e. if $\epsilon(K) = +1$), then $\gamma'_0$ will also turn downward and meets the vertical axis just above height $h_1$; thus $\tau(K_{p,q}) = h_1$. If $\gamma_0$ turns upward (i.e. if $\epsilon(K) = -1$), then $\gamma'_0$ will also turn upwards and meets the vertical axis just below height $h_2$; thus $\tau(K_{p,q}) = h_2 - 1 = p\tau(K)+\frac{(p-1)(q+1)}{2}$. Finally, if $\epsilon(K) = 0$ then $\gamma_0$ agrees with the curve invariant of the unknot, so $\gamma'_0$ is the curve associated with $T_{p,q}$, and thus $\tau(K_{p,q}) = \tau(T_{p,q})$. In particular, if $q > 1$ then $\gamma'_0$ bends down after its first intersection with the vertical axis, and as above $\tau(K_{p,q}) = h_1 = \tfrac{(p-1)(q-1)}{2}$. If $q < -1$ then $\gamma'_0$ bends upward and $\tau(K_{p,q}) = h_2 - 1 = -\tfrac{(p-1)(-q-1)}{2}$, while if $|q| = 1$ then $\gamma'_0$ is horizontal and $\tau(K_{p,q}) = 0$.
\end{proof}

Other concordance invariants can be extracted from $\gamma_0(K)$. For instance, for any positive integer $i$, the invariant $\phi_i(K)$ introduced recently by Dai, Hom, Stoffregen, and Truong \cite{DHST} counts the number of left arcs of $\gamma_0$ of length $i$, where a left arc of length $i$ refers to a segment of $\gamma_0$ connecting successive intersections with the vertical axis whose height differ by $i$ which does not wrap around the cylinder and which lies to the left of the vertical axis. These arcs are counted with sign coming from the orientation of $\gamma_0$, with downward oriented arcs counting positively.\footnote{This is a straightforward translation of the definition of $\phi_i$ given in \cite{DHST} to the language of immersed curves. The \emph{standard complex} described in \cite{DHST} corresponds precisely to the component $\gamma_0$ of $\HFhat(M)$. The integers $\phi_i$ count horizontal arrows of length $i$ in the the standard complex, which correspond to length $i$ right arcs in $\gamma_0$. By symmetry, we can equivalently count length $i$ left arcs in $\gamma_0$.} Like $\tau$, the integers $\phi_i$ are of particular interest in the study of knot concordance because they are additive under connected sum; that is, they define concordance homomorphisms.

%\labellist \tiny
%\pinlabel $(2,-1)$ at 19 8, \pinlabel $(2,1)$ at 124 8
%\endlabellist
%\parpic[r]{
 %\begin{minipage}{55mm}
 %\centering
 %\includegraphics[scale=1.25]{figures/count}
 %\captionof{figure}{An exaggerated view of $\gamma_0(K)$ for an iterated cable of the trefoil $K$, from which we can calculate that $\phi_1(K)=0$.}
% \label{fig:count}
  %\end{minipage}%
  %}
Returning to cabling, the behavior of the invariants $\phi_i$ is more complicated. In particular, $\phi_i(K_{p,q})$ does not depend only on $\phi_i(K)$, or even on the collection of invariants $\tau(K)$, $\epsilon(K)$, and $\phi_j(K)$ for all $j$. In order to express the effect of cabling we need to keep track of how each left arc in $\gamma_0(K)$ behaves at each end. More precisely, we can define refined invariants $\phi_i^{++}$, $\phi_i^{+-}$, $\phi_i^{-+}$, and $\phi_i^{--}$ which encode the signed count of four different types of length $i$ left arcs in $\gamma_0$, with the type determined by the directions $\gamma_0$ turns at each end of the segment; $+$ indicates that $\gamma_0$ turns upward and $-$ indicates that $\gamma_0$ turns downward, with the first sign indicating the behavior at the top of the arc and the second sign indicating the behavior at the bottom of the arc, as in Figure \ref{fig:cases}. Note that $\phi_i = \phi_i^{++} + \phi_i^{+-} + \phi_i^{-+} + \phi_i^{--}$. 
%For example, in the curve shown in Figure \ref{fig:count}, we have that $\phi_1(K)$ is calculated from $\phi_i^{--} + \phi_i^{+-} + \phi_i^{+-}+ \phi_i^{+-}+ \phi_i^{--} + \phi_i^{+-} +\phi_i^{++}+\phi_i^{++}$ where the relevant 8 arcs have been enumerated from top to bottom. The reader can check that in this example we have $\phi_1(K)=0$
With these extra quantities defined it is possible to derive explicit formulas for $\phi_i(K_{p,q})$. However, these formulas are cumbersome and we will not attempt to compute them here in general. As an example, we will consider the special case of $(2,1)$-cabling.%; even in this case, we only give formulas for the variants of $\phi_i$ when $i>1$, as deriving formulas for $i=1$ is more involved.

%\begin{proposition}\label{prop:phi-formulas}
%For $n \ge 1$ we have:
%$$\begin{array}{llll}
%\hspace{2.1mm} \phi_{2n}^{++}(K_{2,1}) = \phi_n^{++}(K) & \hspace{2.1mm}\phi_{2n}^{+-}(K_{2,1}) = 0 & \hspace{2.1mm}\phi_{2n}^{-+}(K_{2,1}) = 0 & \hspace{2.1mm}\phi_{2n}^{--}(K_{2,1}) = \phi_n^{--}(K)\\
%\phi_{2n+1}^{++}(K_{2,1}) = 0 & \phi_{2n+1}^{+-}(K_{2,1}) = \phi_n^{+-}(K) & \phi_{2n+1}^{-+}(K_{2,1}) = \phi_{n+1}^{-+}(K) & \phi_{2n+1}^{--}(K_{2,1}) = 0
%\end{array}$$
%\end{proposition}

%\begin{proposition}\label{prop:phi-formulas}
%For $n \ge 1$ we have:
%\begin{align*}
%\phi_{2n}(K_{2,1})&=\phi_n^{++}(K)+\phi_n^{--}(K) \\
%\phi_{2n+1}(K_{2,1})&= \phi_n^{+-}(K)+\phi_{n+1}^{-+}(K)
%\end{align*}
%\end{proposition}

\begin{proposition}\label{prop:phi-formulas} For $i>1$ all the variants of $\phi_i(K_{2,1})$ are either determined by 
\[\phi_{2n}^{\pm \pm}(K_{2,1})=\phi_n^{\pm\pm}(K) \qquad \phi_{2n \pm 1}^{\pm \mp}(K_{2,1})=\phi_n^{\pm \mp}(K)\] or they are trivial. %and for $i > 1$ all variants of $\phi_i(K_{2,1})$ not determined by these two formulas are trivial. 
In particular, for $n \ge 1$
\begin{align*}
\phi_{2n}(K_{2,1})&=\phi_n^{++}(K)+\phi_n^{--}(K) \\
\phi_{2n+1}(K_{2,1})&= \phi_n^{+-}(K)+\phi_{n+1}^{-+}(K)
\end{align*}
\end{proposition}
\begin{proof}
The curve $\gamma'_0 = \gamma_0(K_{2,1})$ is constructed in three steps: take two consecutive copies of $\gamma_0 = \gamma_0(K)$; scale vertically by a factor of two and shift the second copy of $\gamma_0$ down one unit; and compress horizontally (compare Figure \ref{fig:intro-example}). Before compressing horizontally, we can divide this curve into two (non-connected) subcurves which lie to the left and right of vertical line through the first column of lattice points; let $\gamma_L$ and $\gamma_R$ denote the images of these subcurves after horizontal compression, so that $\gamma'_0 = \gamma_L \cup \gamma_R$.

\labellist \tiny
\pinlabel $\phi_2^{\!-\!-}$ at 14 -5
\pinlabel $\phi_2^{\!-\!+}$ at 36 -5
\pinlabel $\phi_2^{\!+\!-}$ at 60 -5
\pinlabel $\phi_2^{\!+\!+}$ at 84 -5
\endlabellist
\parpic[r]{
 \begin{minipage}{45mm}
 \centering
 \includegraphics[scale=1]{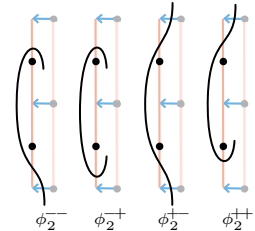}
 \captionof{figure}{Four cases complete the proof.}
 \label{fig:cases}
  \end{minipage}%
  }
The key observation is that every component of $\gamma_R$ lies to the right of every even height lattice point, and therefore any left arc on $\gamma'_0$ which lies in $\gamma_R$ must have length one. It is also clear that any left arc of $\gamma'_0$ intersects at most one component of $\gamma_L$, since otherwise it contains a full component of $\gamma_R$ which must lie to the right of some lattice point. Thus each left arc of length greater than one in $\gamma'_0$ comes from a component of $\gamma_L$, which in turn comes from a left arc of the first copy of $\gamma_0$. Conversely, every left arc of length $i$ in the first copy of $\gamma_0$ gives rise to exactly one left arc in $\gamma'_0$, which has the same end behavior. The length of this new arc depends on the end behavior: it is $2i-1$ for $-+$ arcs, $2i$ for $++$ or $--$ arcs, and $2i+1$ for $+-$ arcs (see Figure \ref{fig:cases}).% The formulas above follow from this observation: one checks that $\phi_{2n}^{+-}(K_{2,1}) = 0$, $\phi_{2n}^{-+}(K_{2,1}) = 0$, $\phi_{2n+1}^{++}(K_{2,1}) = 0$, and $\phi_{2n+1}^{--}(K_{2,1}) = 0$ and the conclusion follows by induction.
\end{proof}

We will say that $\gamma_0$ has a unique maximal length left arc of type $++$ and length $N$ if $\phi_N^{++}(K) = 1$, $\phi_i^{++}(K) = 0$ for all $i > N$, and $\phi_i^{+-}(K) = \phi_i^{-+}(K) = \phi_i^{--}(K) = 0$ for all $i \ge N$. The following is an immediate consequence of the formulas above: 

\begin{proposition}\label{prop:phi-consequence}
If $\gamma_0(K)$ has a unique maximal length left arc of type $++$ and length $N$ then $\gamma_0(K_{2,1})$ has a unique maximal length left arc of type $++$ and length $2N$.
\end{proposition}
Consider for example iterated $(2,1)$-cables of the right handed trefoil $T_{2,3}$; the immersed curves for the first few of these knots are shown in Figure \ref{fig:iterated-cables}. The immersed curve $\gamma_0(T_{2,3})$ has only one left arc, which has type $++$ and length 1. If we repeatedly $(2,1)$-cable this knot, there is always a single left arc of maximal length, which always has type $++$, and the length of this arc doubles in length with each iteration.

\begin{figure}[t]
\includegraphics[scale=1.25]{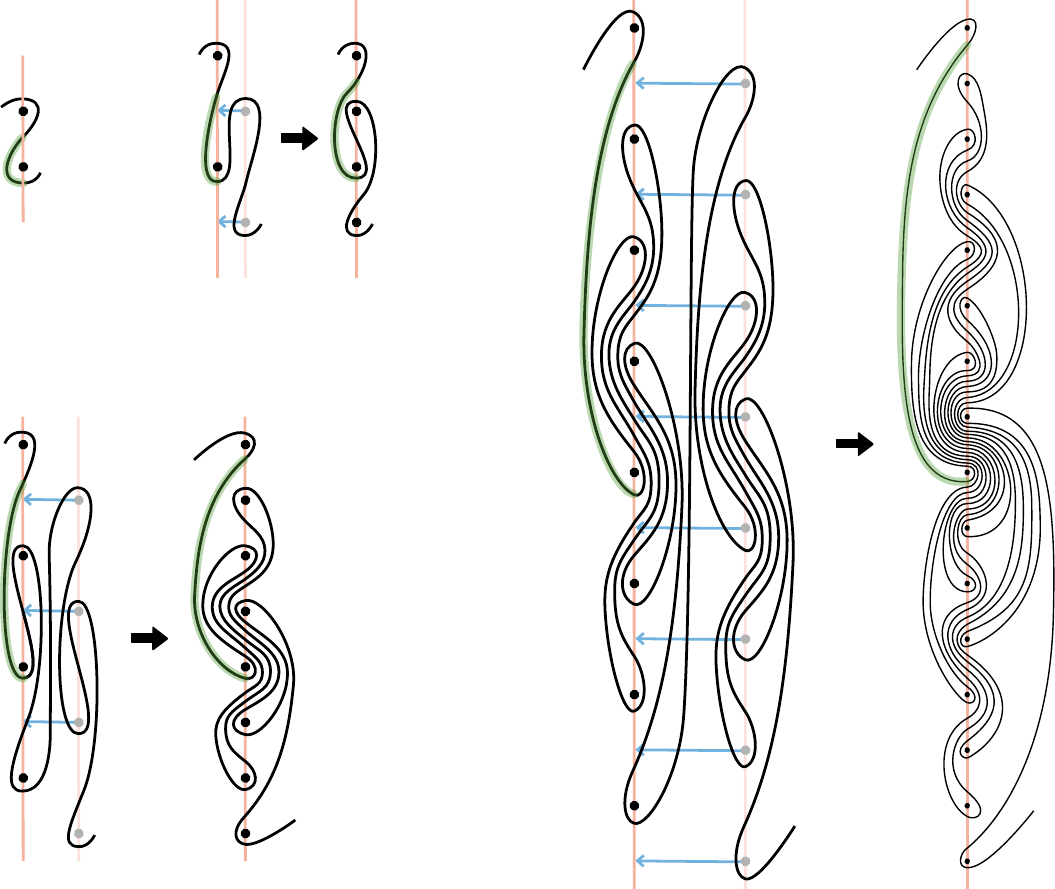}
\caption{Immersed curves for the first few iterated $(2,1)$-cables of the right handed trefoil. These are also the distinguished curve $\gamma_0$ for the knots $K_0$, $K_1$, $K_2$, and $K_3$ from Corollary \ref{cor:Zinfty-summand}. The longest left arc (highlighted) is stretched by a factor of two with each cabling iteration, thus the length of the longest left arc for $K_n$ is $2^n$.}\label{fig:iterated-cables}
\end{figure}

In \cite{DHST}, the concordance invariants $\phi_i$ were used to identify a $\Z^\infty$ direct summand in the topologically slice smooth concordance group; see also \cite{OSS2017}. The relevant infinite family of knots is built from cables of a certain knot $D$, the untwisted positively clasped Whitehead double of $T_{2,3}$. More precisely, the family of knots is given by $D_{n,n+1} \# -T_{n,n+1}$. Using Proposition \ref{prop:phi-consequence}, we can construct another $\Z^\infty$ summand from $D$ by instead taking iterated $(2,1)$-cables. The key properties of $D$ are that $(i)$ the Alexander polynomial of $D$ is trivial and $(ii)$ the distinguished component $\gamma_0$ associated to $D$ agrees with $\gamma_0(T_{2,3})$. The knot $D$ can be replaced with any other knot  which shares these two properties---an example of a hyperbolic knot with this property is $15n{113775}$.

\begin{corollary}\label{cor:Zinfty-summand}
Let $K = K_0$ be a knot for which $\Delta_K(t) = 1$ and $\gamma_0(K) = \gamma_0(T_{2,3})$. For $n \ge 1$ let $K_n$ to be the $(2,1)$-cable of $K_{n-1}$. The knots $\{K_n\}_{n=0}^\infty$ are topologically slice and linearly independent in the smooth concordance group. 
\end{corollary}
\begin{proof}
Applying Freedman, $\Delta_{K}(t) = 1$ implies that $K$ is topologically slice \cite{Freedman1982}. The $(2,1)$-cable of a topologically slice knot is topologically concordant to the $(2,1)$-cable of the unknot, which is the unknot; thus by induction $K_n$ is topologically slice for all $n$. On the other hand, $K_0$ has a unique maximal length left arc of type $++$ and length $1$, so Proposition \ref{prop:phi-consequence} and induction implies that $K_n$ has a unique maximal length left arc of type $++$ and length $2^n$. In particular, $\max\{i | \phi_i(K_n) \neq 0 \} = 2^n$ for all $n$. Since each $\phi_i$ is a concordance homomorphism, it follows that the knots are linearly independent in the smooth concordance group.
\end{proof}

\begin{remark}
We leave the behavior of $\phi_1$ under $(2,1)$ cabling, which was not needed in the above application, as an exercise to the motivated reader, who will find that
$$\phi_1(K_{2,1}) = - \sum_{j \ge 1} \phi_j(K) + \begin{cases} 1 & \text{ if } \tau(K) > 0 \\ 0 & \text{ if } \tau(K) \le 0 \end{cases}$$
Using corresponding formulas for variants of $\phi_1$ along with the formulas from Proposition \ref{prop:phi-formulas} and induction on $n$, it can be shown that the knots $K_n$ in Corollary \ref{cor:Zinfty-summand} in fact satisfy $\phi_i(K_n)=1$ if $i = 2^n$ and $\phi_i(K_n)=0$ otherwise.
\end{remark}

%\begin{figure}[t!]
%\includegraphics[scale=1.25]{figures/iterated-cables-large}
%\caption{Immersed curves for the first few iterated $(2,1)$-cables of the right handed trefoil. These are also the distinguished curve $\gamma_0$ for the knots $K_0$, $K_1$, $K_2$, and $K_3$ from Corollary \ref{cor:Zinfty-summand}. In each case the longest left arc is highlighted in green; this arc is stretched by a factor of two with each cabling iteration, thus the length of the longest left arc for $K_n$ is $2^n$.}\label{fig:iterated-cables}
%\end{figure}

  \parpic[r]{
 \begin{minipage}{20mm}
 \centering
 \includegraphics[scale=1.2]{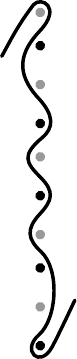}
 \captionof{figure}{}%A figure eight component and $\gamma_0(12n_{242})$, neither of which can arise from a nontrivial cable.}
\label{fig:non-cable-curves}
  \end{minipage}%
  }
\subsection*{Curves which do not come from cables}
In many cases Theorem \ref{thm:main-cable} provides a simple obstruction to a knot being a nontrivial cable. For example, if the curve set corresponding to a knot contains a figure eight component enclosing two adjacent lattice points (as is the case, for example, for any non-torus alternating knot; see \cite{Petkova2013}), then the knot is not a cable of a knot in $S^3$. This is because any closed component of $\HFhat(M_{p,q})$ comes from one copy of $\HFhat(M)$ before the transformation of the plane and thus encloses only lattice points with the same height modulo $p$. If we restrict our attention to $\gamma_0$, we can find knots which are not concordant to a nontrivial cable. An example of this is the knot $12n242$, whose $\gamma_0$ is pictured in Figure \ref{fig:non-cable-curves}. Because the first left arc has length 2,  if this curve comes from a $(p,q)$-cable then $p$ must be at most 3. If $p=1$, the cable operation is trivial. If $p = 2$, then the curve should not pass to the left of an odd height lattice point (gray in figure to the right) after passing to the right of an even height lattice point (black in the figure), but this clearly happens. Similarly, if $p = 3$ then the curve should not pass to the left of a point whose height is congruent to 2 mod 3, after passing to the right of a lattice point with a different height modulo 3, but this also happens. We thank Tye Lidman for asking us about the existence of such an example.

\subsection*{L-space surgeries on cable knots}\label{sub:L-space} Other properties of the knot Floer homology of cable knots are made relatively transparent by Theorem \ref{thm:main-cable}. The following is a well-known property that was established by Hom \cite{Hom2011} (building on work of Hedden \cite{Hedden2009}):

\begin{theorem}[{Hom \cite[Theorem]{Hom2011}}]\label{thm:Lspace}
Suppose $K$ is a knot admitting a positive L-space surgery. Then $K_{p,q}$ admits a positive L-space surgery if and only if  $\tfrac q p$ is at least $2g-1$, where $g$ denotes the Seifert genus of $K$. 
\end{theorem}

  \begin{figure}
   \includegraphics[scale=1]{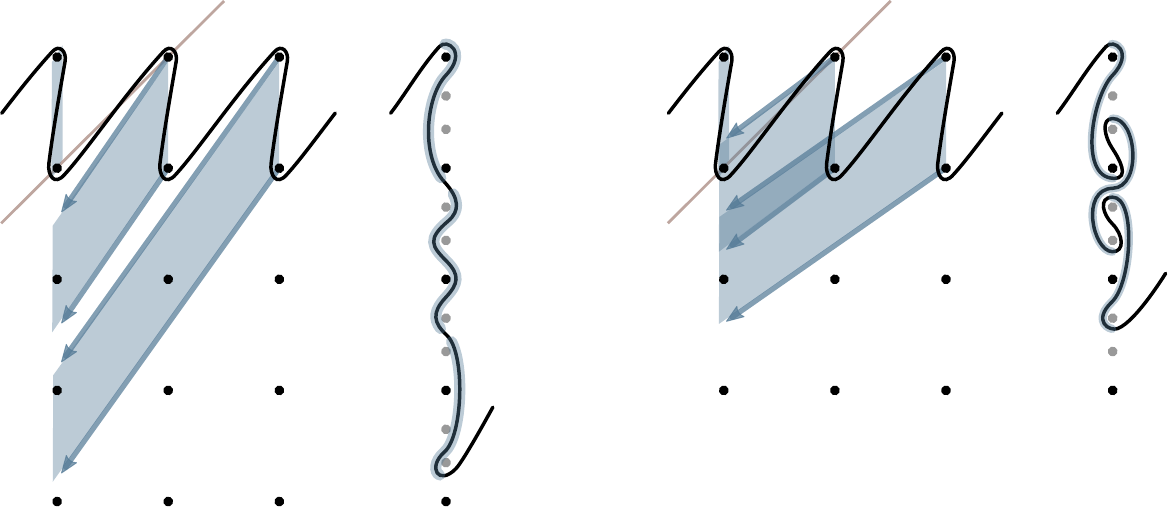}
   \caption{The the (3,4)-cable (left) and the (3,2)-cable (right) of the right hand trefoil.}
 \label{fig:L-space}
   \end{figure}

Note that $K$ admits a positive L-space surgery if and only if $\HFhat(M)$ is a single curve which, apart from the segment that wraps around the cylinder, moves monotonically downward in the neighborhood of the vertical axis (see \cite[Section 7.5]{HRW}). When this curve is pulled tight in the cylinder $\pcylinder$ (or in the plane $\pplane$), the slope of the non-vertical segment is $2g - 1$. Following Theorem \ref{thm:main-cable}, we construct $\HFhat(M_{p,q})$ from $p$ columns of the lift of $\HFhat(M)$ to $\pplane$ by translating lattice points along lines of slope $\tfrac q p$.

 \begin{proof}[A quick re-proof of Theorem \ref{thm:Lspace}]   
If $\HFhat(M)$ is oriented upward at any point apart from the nonvertical segment, it is clear the same will be true at the image of this point on $\HFhat(M_{p,q})$; thus $K$ having a positive L-space surgery is a necessary condition for $K_{p,q}$ to have one. Supposing $K$ has a positive L-space surgery, it is clear that if $\tfrac q p > 2g-1$ then the $p$ copies of the downward oriented portion of $\HFhat(M)$ miss each other, so that the resulting curve moves monotonically downward and $K_{p,q}$ has a positive L-space surgery. On the other hand, if $\tfrac q p < 2g-1$ then these sections of curves overlap, forcing some backtracking in the resulting curve, implying that $K_{p,q}$ has no L-space surgeries. An example is given in Figure \ref{fig:L-space}.
 \end{proof}

\subsection*{Cabling via tiling} From Theorem \ref{thm:main-cable} it is possible to interpret cabling in terms of plane tilings. That is, in a visual summary of the above discussion, we record the following:

\begin{figure}[t!]
\includegraphics[scale=1.25]{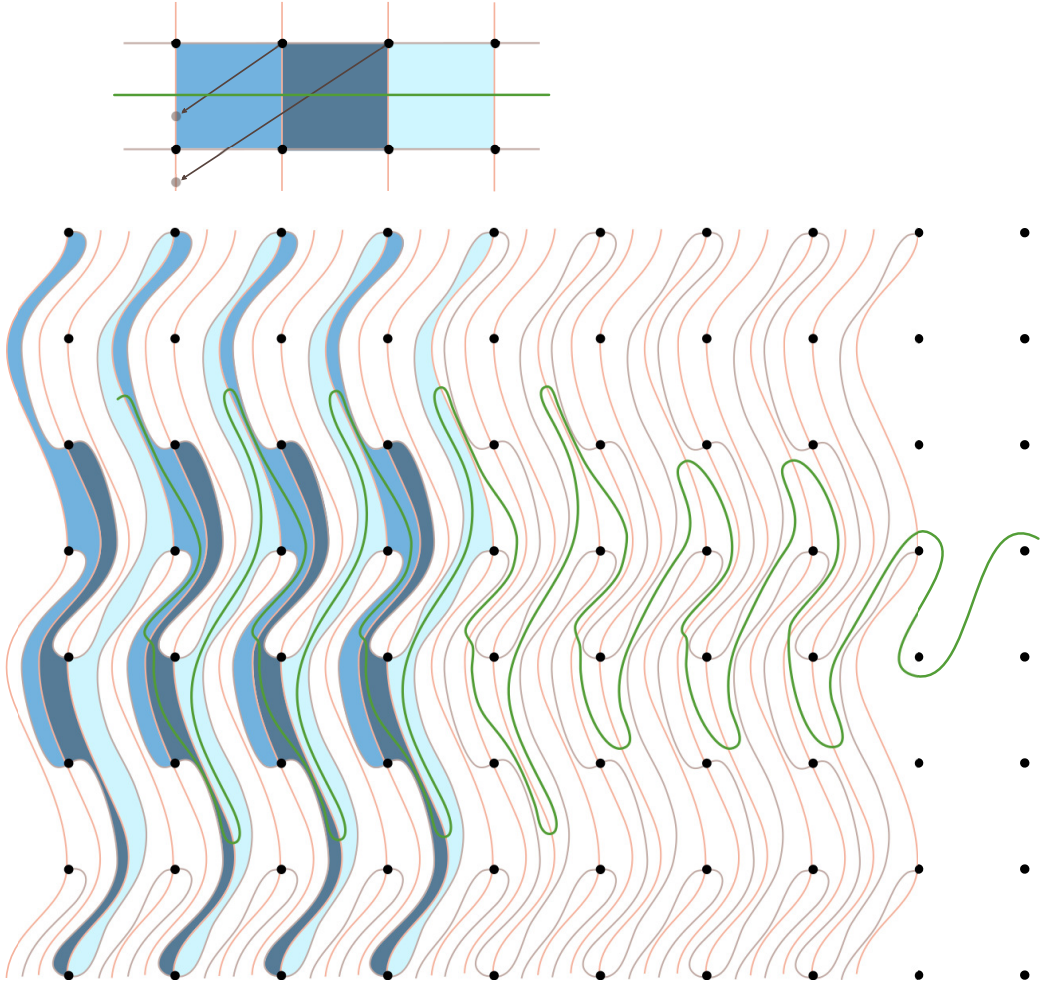}
\caption{The $(3,2)$-cabling operation interpreted as a plane tiling: 3 copies of the standard square tile (above) are carried to a new regular tile in $\mathcal{T}_{3,2}$ (below) under the operation $\boldsymbol f_{p,q}$ appearing in Theorem \ref{thm:main-cable}. To illustrate, the image of the longitude has been included (gradually homotoped to a simpler form moving righward), which recovers the invariant associated with the right-hand trefoil as expected.}\label{fig:intro-tiles}
\end{figure}

\begin{corollary}
For every relatively prime pair $(p,q)$ there is a periodic tiling $\mathcal{T}_{p,q}$ of the plane, unique up to lattice-fixing planar isotopy, so that $\curveset(K_{p,q})$ is the image of $\curveset(K)$ under the transformation taking the lattice $\widetilde{T}_\bullet$ to $\mathcal{T}_{p,q}$.
\end{corollary}

\begin{proof} This is a simple reformulation of Theorem \ref{thm:main-cable}: Consider the standard square tiling of the plane  $\widetilde{T}_\bullet$ defined by the preferred $(\mu,\lambda)$-framing. The image of $p$ square tiles aligned horizontally, under the  application of $\boldsymbol f_{p,q}$, gives a tile in a periodic tiling of the plane. %Note that, in order to express the resulting tiling relative to the standard square lattice (that is, the longitude and meridian of the cable in question), there is a vertical contraction and a horizontal expansion that needs to be scaled away.  
\end{proof}

This is best illustrated in an example, and we have shown the tiling associated with $(3,2)$-cabling in Figure \ref{fig:intro-tiles}. Note that this point of view comes with a built-in sanity check: One can check that the image of a longitudinal curve under the transformation to $\mathcal{T}_{p,q}$ is the immersed curve $\curveset(T_{p,q})$. Recalling that, as a polynomial in $t$, the Alexander polynomial satisfies
\[\Delta(K_{p,q}) =\Delta(K)|_{t^p}\cdot \Delta(T_{p,q}) \] our formula has $\mathcal{T}_{p,q}$ playing the role of $\Delta(T_{p,q})$ in this formula while replacing $t$ with $t^p$ corresponds to the $p$ repeated copies of $\curveset(K)$.

In general, one expects bimodules in bordered Floer homology (for manifolds with two boundary tori) to be associated with 
Lagragians in $\ptorus \times \ptorus$. A simple first example of this is the bimodule associated with a diffeomorphism of the torus, where the (embedded) Lagrangian surface is the graph of the diffeomerophism. In that case we can interpret the action of the bimodule as follows: to compute the image of an immersed curve $\gamma$, we consider $\gamma \times \ptorus$, intersect with the Lagrangian surface, and project to the second coordinate. Cabling bimodules provide a first glimpse at how this construction might be generalized to arbitrary bimodules. The diffeomeorphism of the plane $\boldsymbol f_{p,q}$ does not descend to a diffeomorphism of torus, but since $\boldsymbol f_{p,q}$ is periodic and is determined by its effect on $p$ consecutive tiles of the plane it can be viewed as a $p$-valued function on $\ptorus$; that is, to each point in $\ptorus$ it associates an unordered tuple of $p$ points in $\ptorus$. The graph of this multivalued function is an (immersed) Lagrangian surface in $\ptorus \times \ptorus$, and the action of the bimodule on curves can be interpreted geometrically as before.

\section{Immersed curves and the merge operation}\label{sec:curves}% !TEX root = ../cabling.tex

% \section{Immersed curves and the merge operation}
% \label{sec:curves}

For any orientable manifold $M$ with torus boundary the Heegaard Floer homology $\HFhat(M)$ is an immersed multicurve in the marked torus $\partial M$ \cite{HRW}, as introduced above. This view of the Heegaard Floer invariants of $M$ arises from an interpretation of bordered Floer homology \cite{LOT}, and is closely related to the {\em loop calculus} introduced in \cite{HW}. This section builds a glossary between loop calculus \cite{HW} and immersed curves \cite{HRW}; in the former we developed the machinery for understanding gluing pairs of manifolds along essential annuli in their boundaries which we aim to interpret in terms of the immersed curves in the case of cabling knots in the three-sphere.  %presumably the rest of this boilerplate will be in the introduction. Really I think all we want to say is that there is \cite{HW} and \cite{HRW}, and we need to explain how to move between the two in order to merge (particularly, glue a solid torus to some $M$ along an essential annulus). 

\subsection{From puzzle pieces to curve segments} Assuming familiarity with some subset of \cite{HRW,HW,LOT}, we give a very terse summary of the bordered invariants in order to set up the desired glossary. 

\labellist \small
%\pinlabel $\rho_2\rho_1=0$ at 120 15
%\pinlabel $\rho_3\rho_2=0$ at 120 5
%\pinlabel $\iota_0$ at -5 49 \pinlabel $\iota_1$ at 122 49
%\pinlabel $\rho_1$ at 59 57
%\pinlabel $\rho_2$ at 59 37
%\pinlabel $\rho_3$ at 59 17
\pinlabel $\rho_2\rho_1=0$ at 78 13
\pinlabel $\rho_3\rho_2=0$ at 78 3
\pinlabel $\iota_0$ at -6 38 \pinlabel $\iota_1$ at 65 38
\pinlabel $\rho_1$ at 10 57
\pinlabel $\rho_2$ at 22 43
\pinlabel $\rho_3$ at 10 18
\endlabellist
\parpic[r]{
 \begin{minipage}{55mm}
 \centering
 \includegraphics[scale=1.0]{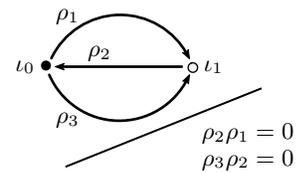}
 \captionof{figure}{The torus algebra $\Alg$ as the path algebra of a quiver with relations.}
 \label{fig:quiver}
  \end{minipage}%
  }
The torus algebra $\Alg$ is obtained as the path algebra of the quiver described in Figure \ref{fig:quiver}. Let $\sI\subset\Alg$ denote the subring of idempotents generated by $\iota_0$ and $\iota_1$. Working over the two-element field $\F$, a type D structure over $\Alg$ is a finite dimensional left $\sI$-module $V$ together with a map $\delta\co V\to\Alg\otimes_\sI V$. This map must satisfy a compatibility condition equivalent to ensuring that $\partial(a\otimes x)=a\cdot \delta(x)$ is a differential on the $\Alg$-module $\Alg\otimes_\sI V$. %Note that a type D structure over $\Alg$ is the same thing as a twisted complex over $\Alg$.

There is a simple interpretation of the above data in terms of decorated graphs: the vertices encode the generating set (these come in two types $\bullet$ and $\circ$, depending on the idempotents $\iota_0$ and $\iota_1$, respectively) and, by passing to type D structures that are reduced, the directed edges are labelled by the set $\{1,2,3,12,23,123\}$ in order to encode the coefficient maps; see Figure \ref{fig:curves-v-loops}. These graphs can be naturally immersed in the marked torus or, more precisely, in the once punctured torus with a fixed choice of 1-handle co-cores cutting the surface into a disk. In our case, these co-cores will always coincide with the preferred  $(\mu,\lambda)$-pair, since we are focused on knots in $S^3$. With this data in hand, we can decompose the torus into the familiar square patch with opposite edges identified. The type D structures of interest then are immersed train tracks (in the sense of Thurston \cite{Mosher2003}) where all of the vertices/switches lies on the horizontal or vertical edges; when such a train track comes from a three-manifold the classification theorem proved in \cite{HRW} tells us it is equivalent to an immersed multi-curve, possibly decorated with local systems, which we denote by $\HFhat(M)$ \cite{HRW}. %Again, for the purpose of this paper, $M$ is always the exterior of some knot $K\into S^3$. 

\labellist \tiny
\pinlabel $a_{k}$ at 21 5
\pinlabel $b_{k}$ at 87 6
\pinlabel $c_{k}$ at 150 5
\pinlabel $d_{k}\!=\!\bar{c}_{\text{-}k}$ at 213 6
\pinlabel $e\!=\!\bar{c}_0$ at 276 5
\tiny  
\pinlabel $3$ at 17 18 \pinlabel $23$ at 27 42 \pinlabel $2$ at 19 80 
  \pinlabel $123$ at 83 18 \pinlabel $23$ at 78 42 \pinlabel $1$ at 86 80 
 \pinlabel $3$ at 146 18 \pinlabel $23$ at 154 42 \pinlabel $1$ at 152 80 
  \pinlabel $123$ at 212 18 \pinlabel $23$ at 217 42 \pinlabel $2$ at 209 80
   \pinlabel $12$ at 281 52.5  
\endlabellist
\begin{figure}[t!]
\includegraphics[scale=1.0]{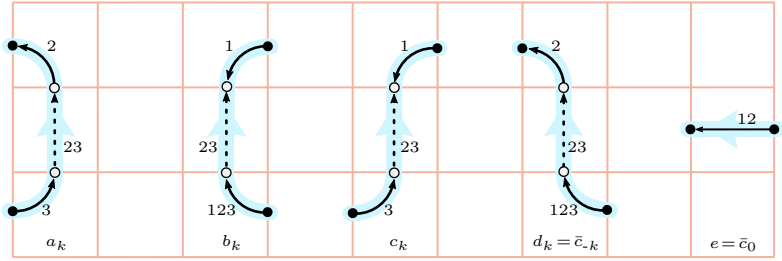}
\caption{Segments of immersed curve in the cover of the marked torus, labelled to be consistent with the puzzle pieces given in \cite{HW}. The integer subscript $k>0$ indicates the number of $\circ$ generators in the segment. These letters can appear forwards or backwards in a cyclic word, so that $\bar a_k$ runs against the direction indicated by the blue arrow. We can also extend our notation by setting $a_{-k} = \bar a_k, b_{-k} = \bar b_k$, $c_{-k} = \bar d_k$, $d_{-k} = \bar c_k$ and $d_0 = \bar c_0 = e$; note that then a segment with subscript $k$ moves upward $k$ units in the plane. %Type $a$ and $b$ segments have no net horizontal movement, while type $c$ segments move rightward by one.
}\label{fig:abcd}
\end{figure}

\labellist \small
\pinlabel $\mu$ at 2 81
\pinlabel $\lambda$ at 83 1
\tiny  
\pinlabel $z$ at 80 78
\pinlabel $1$ at 70 24 \pinlabel $1$ at 53 24 \pinlabel $23$ at 37 24 \pinlabel $2$ at 27 24
\pinlabel $3$ at 27 45 \pinlabel $3$ at 27 70 
  \pinlabel $123$ at 50 70
   \pinlabel $\boldsymbol 3$ at 113 84    \pinlabel $\boldsymbol 2$ at 113 12    \pinlabel $\boldsymbol 1$ at 184 12
 \endlabellist
\begin{figure}[h]
\includegraphics[scale=1.0]{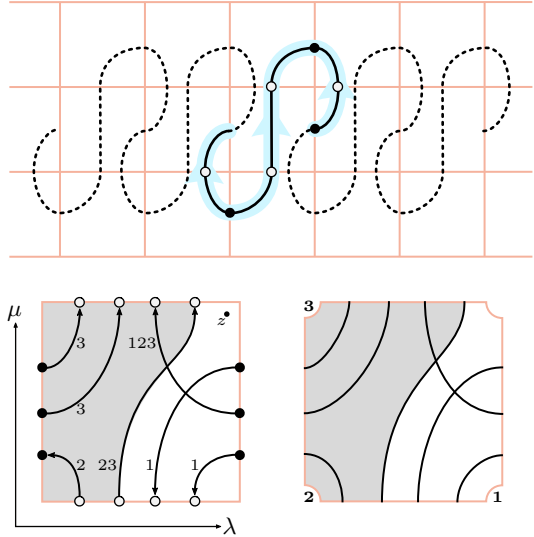}
\caption{Three different views of the invariant associated with the exterior of the right-hand trefoil. In all three cases, we have fixed the preferred $(\mu,\lambda)$ framing in order to present the torus boundary. On the lower left, the decorated graph describing the type D structure has been immersed in the marked torus as a train track. This description exhibits the redundancy in the edge labels: as shown in the lower right figure, the idempotents can be recovered from the horizontal and vertical edges while the coefficient maps are determined by which of the labelled corners are traversed by the curve segments (a region indicating the $23$ edge is shaded). Finally, lifting the curve to the cover $\overline{T}_\bullet$ (or, as pictured, $\pplane$) makes obvious the cyclic word $a_1\bar{c}_2b_1$, which in \cite{HW} is referred to as a {\em loop}.}\label{fig:curves-v-loops}
\end{figure}

In the case where the local systems are trivial, we recover the class of loop-type manifolds considered in our earlier work \cite{HW} (see also \cite[Section 1]{HRW-companion}). Central to this is the observation that, when the type D structure in question can be represented by a valence 2 graph, it is possible to decompose along $\bullet$ vertices into segments, each of which takes one of 5 possible forms as described in Figure \ref{fig:abcd} (compare \cite[Figure 1]{HW}).\footnote{There are certain exceptional type D stuctures that cannot be decomposed in this way, however these examples are not particularly important in this setting. The interested reader can consult \cite{HW} for a {\em dual} notation that decomposes along $\circ$ vertices.} As a result, studying these type D structures amounts to a calculus for manipulating cyclic words in the infinite alphabet $\mathfrak{A}=\{a_k,b_k,c_k,d_k, e\}$ for all positive integers $k$. The segments corresponding to these letters may appear backwards as we traverse a loop; this is indicated by a \emph {bar}. There are rules governed by the algebra restricting the letters that can be concatenated, which are most easily described by noting that each segment also corresponds to a segment of immersed curve as in Figure \ref{fig:abcd}: if two curve segments share an endpoint, they must lie on opposite sides of the vertical near that point. (In \cite{HW}, a puzzle piece convention is used to describe these rules.) Note that the $a_k$ and the $b_k$ correspond to the two types of stable chains introduced in \cite{LOT}, while $c_k$, $d_k$, and $e$ correspond to the three types of unstable chains. In fact, it makes sense to view the three types of unstable chain as part of a single family, and with this in mind we set $c_0 = \bar{e}$ and $c_{-k} = \bar{d}_k$. The example in Figure \ref{fig:curves-v-loops} explains this for the right-hand trefoil exterior.

\labellist \small
\pinlabel $\cdots$ at -7 80
\pinlabel $\cdots$ at -7 18
\pinlabel {$\left(\F^3, \left(\begin{smallmatrix}0&0&1\\1&0&0\\1&1&1\end{smallmatrix}\right)\right)$} at 44 48
\tiny
\pinlabel $x$ at 19 29.5 \pinlabel $y$ at 19 68 
\pinlabel $x_3$ at 81 32.5 \pinlabel $x_1$ at 107 32.5
\pinlabel $y_3$ at 81 64 \pinlabel $y_1$ at 107 64
\endlabellist
\parpic[r]{
 \begin{minipage}{55mm}
 \centering
 \includegraphics[scale=1.0]{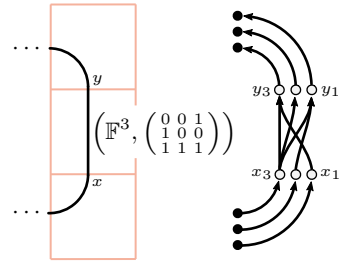}
 \captionof{figure}{A 3-dimensional local system, expanded at an $a_2$ to give a train track.}
 \label{fig:track-like}
  \end{minipage}%
  }
It is relatively straightforward to extend this language to cases admitting a non-trivial local system. Recall that each letter in $\mathfrak{A}$ corresponds to a (portion of a) type D structure that is a linear chain of arrows. We allow a letter in $\mathfrak{A}$ to be decorated by a local system, as follows. Let $V$ be a vector space over $\F$ of dimension $n$, and let $\Phi\co V\to V$ be an endomorphism. Decorating a segment with $(V, \Phi)$ amounts to taking $n$ parallel copies of the appropriate chain, with the $n$ parallel copies of any one arrow in the chain replaced with a collection of arrows determined by $\Phi$; see Figure \ref{fig:track-like} for an example. Suppose the relevant arrow in the chain connects generators (i.e. vertices) $x$ and $y$, with $x$ occurring first in the chain (i.e. there is an edge connecting $x$ to $y$). Fix bases $\langle x_1,\ldots,x_n\rangle$ and $\langle y_1,\ldots,y_n\rangle$ for $V$ and consider an $n\times n$ matrix over $\F$ representing $\Phi$, these arrows connect the $i$th copy of $x$ (i.e. $x_i$) to the $j$th copy of $y$ (i.e $y_j$) if and only if the $(i,j)$ entry of the matrix is nonzero (again, see Figure \ref{fig:track-like}). If all letters in a cyclic word carry a local system (each having the same dimension), then the local system on the cyclic word is determined by composing the endomorphisms. Note that a letter decorated by the trivial local system of dimension $n$ corresponds to $n$ parallel copies of the relevant curve segment. %In particular, to represent a closed loop with a given local system we can always assume a single letter in the cyclic word carries this local system and all others carry trivial local systems of the same dimension.

Now consider a component of $\HFhat(M)$, that is, an immersed curve $\gamma$ decorated with a local system $(V, \Phi)$ of dimension $n$. Following \cite{HRW}, we can interpret this as a \emph{curve-like train track}, which consists of $n$ parallel copies of $\gamma$ along with some additional edges that we may assume all lie on a portion of $\gamma$ corresponding to a single segment (that is, along one letter of $\mathfrak{A}$ as described above). When $M$ is the complement of a knot $K$ in $S^3$, we may in fact assume that these edges lie on a segment of type $a_k$; this is because the curve $\gamma_0(K)$ does not carry a nontrivial local system, and all other curves are closed in the lift to the plane $\pplane$ and thus must contain a type $a_k$ segment. The portion of the train track containing the extra edges is precisely a type $a_k$ segment with local system $(V, \Phi)$. These extra edges determine an $n\times n$ matrix over $\F$, where the $(i,j)$ entry is nonzero if the curve-like train track contains a copy of the $a_k$ segment from the $i$th copy of the initial generator in the segment to the $j$th copy of the final generator of the segment. By construction, this matrix represents the local system $\Phi$.

To summarize, given a knot $K$ in $S^3$, the invariant $\HFhat(M)$ is an immersed multi-curve $\curveset(K)=(\gamma_0,\gamma_1, \ldots,\gamma_n)$ where each $\gamma_{i>0}$ caries a (possibly trivial) local system. By the above discussion, we can assume that each component of $\curveset(K)$ is represented by a cyclic word in $\mathfrak{A}$, possibly with a nontrivial local system on a single $a_k$ segment. 
  
%It is relatively straightforward to extend this language to cases admitting a non-trivial local system. In light of the classification theorem proved in \cite{HRW}, following the notation introduced in the introduction, a given type D structure is a multicurve $\curveset=(\gamma_0,\gamma_1, \ldots,\gamma_n)$ where the $\gamma_{i>0}$ are compact in the lift to $\widetilde{T}_\bullet$. These $\gamma_{i>0}$ may carry non-trivial local systems; owing to the fact that they are compact in $\widetilde{T}$, any such curve has a segment modelled on a letter $a_k$ for some $k$.  This allows us to restrict to train tracks with a particularly nice form, which we refer to as {\em curve-like train tracks}. To obtain such a train track, consider a $\gamma_i$ with local system $(V,\Phi)$; given a basis for $V$, the map $\Phi\co V\to V$ is represented by a square matrix with entries that are $0$ or $1$ (recall that we work over the 2 element field throughout).  Now we construct a train track in two steps: first take $\dim (V)$ parallel copies of $\gamma_i$, and second we replace the parallel segments in a choice of $a_k$ in the curve $\gamma_i$ with a collection of switches according to the matrix entries. See Figure \ref{fig:track-like}. With the exception of a single $a_k$ letter (which we have not complicated to encode $\Phi$ and denote now by $\boxed{a}$), such a train track can be studied locally in terms of words from the alphabet $\mathfrak{A}$. When non-trivlal local systems arise, we will always make use of a curve-like train track to encode the associated type D structure. 

\subsection{The merge operation} Given type D structures $\boldsymbol\vartheta$ and $\boldsymbol\gamma$, we describe a new type D structure $\me(\boldsymbol\gamma,\boldsymbol\vartheta)$. This follows the notation set out in \cite{HW}, where we showed that this type D structure agrees with $\HFhat(\mathcal{M}(M_1,M_2))$ in the case where $\boldsymbol\gamma$ and $\boldsymbol\vartheta$ correspond to $\HFhat(M_1)$ and $\HFhat(M_2)$, respectively.\footnote{More specifically, treating $\boldsymbol\gamma$ and $\boldsymbol\vartheta$ as type D structures, in \cite{HW} we use $\me(\boldsymbol\gamma,\boldsymbol\vartheta)$ as a shorthand for the type D structure $\widehat{\mathit{CFDAA}}(\sP\times S^1)\boxtimes(\boldsymbol\gamma,\boldsymbol\vartheta)$, where $\widehat{\mathit{CFDAA}}(\sP\times S^1)$ is the bordered trimodule calculated in \cite{Hanselman2016} and the three-manifold $\sP\times S^1$ is a circle bundle over a pair of pants (this plays a key role in the next section).} The operations $\me$ and $\mathcal{M}$ are referred to as merges; the latter glues two manifolds along essential annuli in their torus boundaries. We will first describe the operator $\me$ algebraically, and then explain the gluing conventions for $\mathcal{M}$ in the next section in the context of cabling.

Some simplifications are possible in the present setting. First, we assume that $\boldsymbol\gamma$ is a loop consisting of only of $c_k$'s for integers $k$. This assumption holds in particular when $M_1$ is a solid torus, in which case $\mathcal{M}(M_1,-)$ will give rise to a cabled knot.\footnote{In fact, everything we do works in a much more general setting: Any manifold admitting L-space surgeries has a type D structure that, relative to a slope corresponding to an L-space filling, can be expressed in terms of only letters $c_k$. We have opted to simplify matters and focus on a well known construction with well-established conventions in order to illustrate the key principle. More general cases follow the same lines, and can be extracted from \cite{HRW,HW}.}  Further, as described above, we assume that $\boldsymbol\vartheta$ is represented as a curve-like train track. This may consist of several disjoint components, but we can restrict to connected train tracks without loss of generality: If $\boldsymbol\vartheta = (\vartheta_0,\ldots,\vartheta_n)$  then $\me(\boldsymbol\gamma,\boldsymbol\vartheta) =\left(\me(\boldsymbol\gamma,\vartheta_0),\ldots, \me(\boldsymbol\gamma,\vartheta_n)\right)$.

%We consider the case where $\gamma$ and $\gamma'$ are connected train tracks: If $\boldsymbol\vartheta = \cup_{i=1}^{n}\vartheta$  then $\me(\boldsymbol\gamma,\boldsymbol\vartheta) = \cup_{i=1}^{n}\me(\boldsymbol\gamma,\vartheta_i)$. Further, assume that $\boldsymbol\gamma$ is a loop consisting of only of $c_i$'s where $i$ runs over all integers. This latter assumption is sufficient $M_1$ is a solid torus, so that $\mathcal{M}(M_1,-)$ will give rise to a cabled knot.\footnote{In fact, everything we do works in a much more geneal setting which we should say once here, in this footnote, and exaplain once in the introduction. We're making a choice: I think it is easier for peope to see what is going on by restricting to an operation that is familiar and has standard notation.} 

\labellist
  \pinlabel $\boldsymbol\vartheta$ at 19 149
    \pinlabel $\boldsymbol\gamma$ at 125 276
   \pinlabel $\lfloor\boldsymbol\gamma\rfloor$ at 225 35 
 \small
 \pinlabel $1$ at 251 360 
 \pinlabel $2$ at 282 360 
  \pinlabel $-1$ at 316 360 
    \pinlabel $1$ at 353 360 
    \pinlabel $2$ at 388 360 
  \pinlabel $-1$ at 417 360 
\tiny  \pinlabel $c_{1}$ at 123 235 \pinlabel $c_{2}$ at 155 235  \pinlabel $c_{\text{-}1}$ at 187 235
\pinlabel $c_{1}$ at 93 48 \pinlabel $b_{1}$ at 93 76 \pinlabel $c_{0}$ at 93 107  \pinlabel $a_{2}$ at 93 140  \pinlabel $c_{1}$ at 93 175  \pinlabel $c_{1}$ at 93 210
%rows
\pinlabel $c_{2}$ at 125.5 210 \pinlabel $c_{3}$ at 158.5 210 \pinlabel $c_{0}$ at 190 210
\pinlabel $c_{2}$ at 125.5 178 \pinlabel $c_{3}$ at 158.5 178 \pinlabel $c_{0}$ at 190 178
\pinlabel $a_{2}$ at 123 140 \pinlabel $a_{2}$ at 155.5 140 \pinlabel $a_{2}$ at 187 140
\pinlabel $c_{1}$ at 125.5 105 \pinlabel $c_{2}$ at 158.5 105  \pinlabel $c_{\text{-}1}$ at 191.5 105
 \pinlabel $b_{1}$ at 123 76 \pinlabel $b_{1}$ at 155 76 \pinlabel $b_{1}$ at 187 76
\pinlabel $c_{2}$ at 125.5 50 \pinlabel $c_{3}$ at 158.5 50 \pinlabel $c_{0}$ at 190 50
   \endlabellist
\begin{figure}[t]
\includegraphics[scale=0.9]{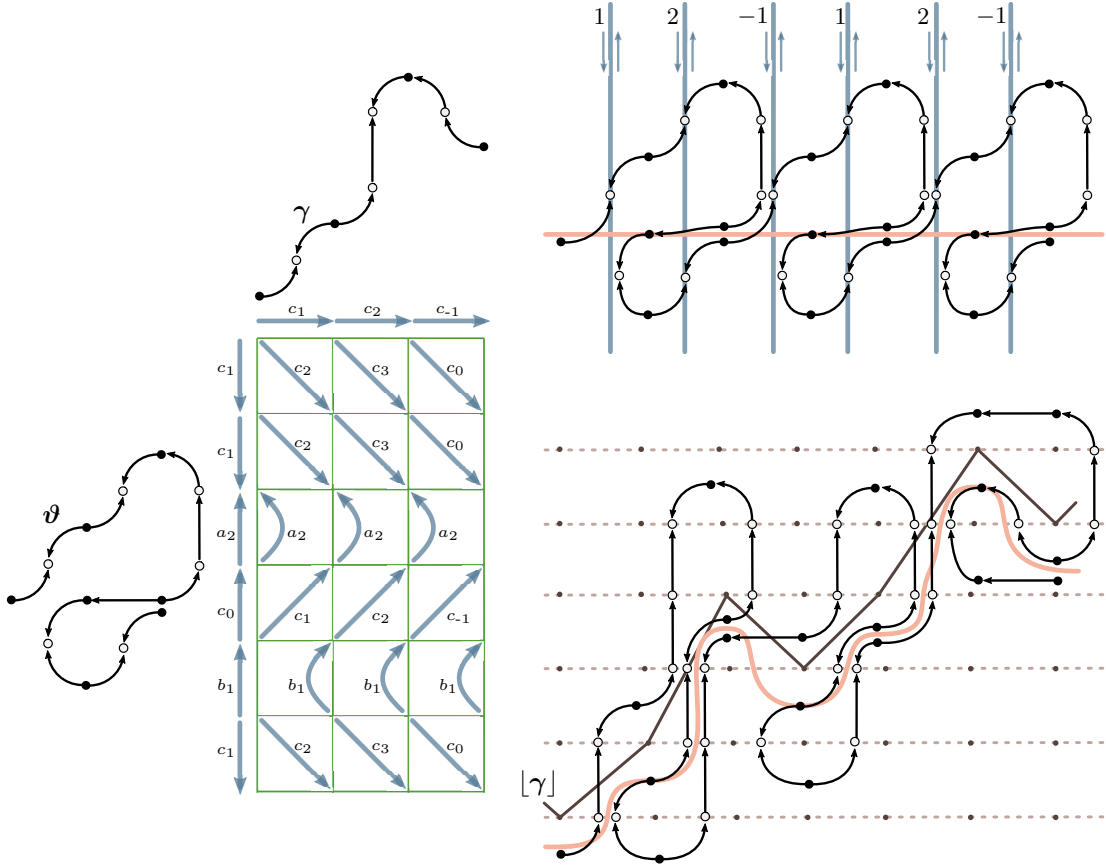}
\caption{Merging a pair of curves, as in Proposition \ref{prp:merge}, described graphically: On the left hand side of the diagram, the output curve is interpreted on a toroidal grid, where the $c_k$ from $\boldsymbol\gamma$ (written on the horizontal) act on the letters in $\boldsymbol\vartheta$ (written on the vertical). On the right, this process in interpreted in terms of curves, where the top right figure give a section of the (periodic) curve in $\pplane$ while the bottom right figure is the result of the merge. Note that the horizontal is moved to the {\em key curve} $\lfloor\boldsymbol\gamma\rfloor$.}\label{fig:merge}
\end{figure}

The main tool used in this paper is a distilled version of \cite[Proposition 6.4]{HW}:

\begin{proposition}\label{prp:merge}
Let $\boldsymbol\vartheta$ be a type D structure represented by a single cyclic word in $\mathfrak{A}$ and let $\boldsymbol\gamma$ be a word containing only $c_k$'s. If the local system on $\boldsymbol\vartheta$ is trivial then the type D structure $\me(\boldsymbol\gamma,\boldsymbol\vartheta)$ is obtained by applying the rules \[\me(a_j,c_k)=a_{j} \quad \me(b_j,c_k)=b_{j} \quad \me(c_j,c_k)=c_{j+k}\] to every letter in $\boldsymbol\vartheta$, ranging over all letters for $\boldsymbol\gamma$, and assembling the result together using a toroidal grid to match up the endpoints, as described in Figure \ref{fig:merge}. %In particular, if $\boldsymbol\vartheta$ is a curve then $\me(\boldsymbol\gamma,\boldsymbol\vartheta)$ is a curve; and if $\boldsymbol\vartheta$ is a curve-like train-track then $\me(\boldsymbol\gamma,\boldsymbol\vartheta)$ is a curve like train-track.
\end{proposition}

The proof of this result is contained in \cite{HW}, however because non-trivial local systems are not handled there, we want to be precise about how to extend the result based on the material in our earlier work. 

\begin{proposition}\label{prp:merge-general}
Proposition \ref{prp:merge} holds when $\boldsymbol\vartheta$ carries a local system, where for each letter $u$ in the word representing $\boldsymbol\vartheta$ and each $c_k$ in the word representing $\boldsymbol\gamma$, $\me(u,c_k)$ carries the same local system as $u$.
\end{proposition}

\begin{proof}
There is nothing to check for trivial local systems, as these are just disjoint copies of some curve and Proposition \ref{prp:merge} applies. For a nontrivial local system, we need to carry out the computation in \cite[Figures 10, 11, and 12]{HW}, replacing the simple segments for the $\boldsymbol\ell_2$ input with segments carrying an arbitrary local system as in, for example, Figure \ref{fig:track-like}. This is a straightforward computation. Note that we do not need to check this computation for type $b_k$ pieces, since we may assume a local system on a loop is concentrated on any one letter and a loop containing a $b_k$ must also contain an $a_j$. (Furthermore if $\boldsymbol\vartheta$ is the multicurve corresponding to a knot in $S^3$, as in this paper, it is enough to check the computation for $a_k$ pieces since any component with a nontrivial local system must contain an $a_k$ piece.)
\end{proof}

%Notice that it follows that local systems on (a component) of $\boldsymbol\vartheta$ carry over trivially to the same local systems on the relevant components of $\me(\boldsymbol\gamma\boldsymbol\vartheta)$. 

%In order to extend the result, it is enough to consider the case when $\vartheta$ is a curve-like train track. Tracing through the proof in \cite{HW} (see also \cite[Section 5]{HRW}), the reader will see that everything goes through without changes on every letter other than $\boxed{a}$. In fact, the same proof will show that $\me(\boxed{a},c_k)$ gives a segment homotopy equivalent to the segment $\boxed{a}$, and the desired result holds. In particular, while the underlying curve for  $\me(\boldsymbol\vartheta,\boldsymbol\gamma)$ changes, the local system remains the same. 

 The takeaway from Proposition \ref{prp:merge} (and its extension to non-trivial local systems in Proposition \ref{prp:merge-general}) is a graphical calculus used to determine the merge of two curves when one contains only $c_k$ segments; this is the content of Figure \ref{fig:merge}. Consider a word (in the $c_k$) representing $\boldsymbol\gamma$, and write this along the top of a rectangular grid; consider a word in $\mathfrak{A}$ representing $\boldsymbol\vartheta$, and write this along the side of the rectangular grid. Then following the letter-by-letter instructions in Proposition \ref{prp:merge}, the new word $\me(\boldsymbol\gamma,\boldsymbol\vartheta)$ can be obtained by running through the grid, starting at the top left. As shown in Figure \ref{fig:merge}, the $\me(c_j,c_k)=c_{j+k}$ run diagonally, while the $\me(a_j,c_k)=a_{j}$ and $\me(b_j,c_k)=b_{j}$ change direction. The sides of this grid are identified to form a toroidal grid, and this connects up the endpoints of the segments to form the new loop $\me(\boldsymbol\gamma,\boldsymbol\vartheta)$.

 \labellist \small
\pinlabel $\lceil\boldsymbol\gamma\rceil$ at -8 35
\pinlabel $\boldsymbol\gamma$ at -5 19
\pinlabel $\lfloor\boldsymbol\gamma\rfloor$ at -8 1
\endlabellist
\parpic[r]{
 \begin{minipage}{45mm}
 \centering
 \includegraphics[scale=0.80]{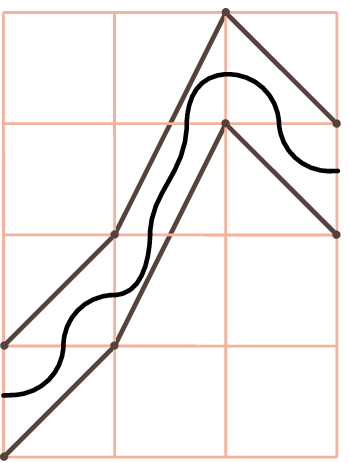}
 \captionof{figure}{The PL {\it key curves} approximating $\curveset$.}
 \label{fig:floor-and-ceiling}
  \end{minipage}%
  }
 Note that as the resulting loop is traversed, horizontal motion in the grid corresponds exactly to horizontal motion of the corresponding curve in the plane. In particular, each vertical line of lattice points in the plane corresponds to a column of the grid containing some $c_k$ in $\boldsymbol\gamma$, and the effect of merging on the curve for $\boldsymbol\vartheta$ is to shear the plane along that vertical line by $k$.

The graphical shorthand from Figure \ref{fig:merge} suggests an interpretation of the merge operation in terms of immersed curves, which we can think of as $\boldsymbol\gamma$ acting on $\boldsymbol\vartheta$. To describe this,  it is useful to have a piecewise-linear representative of the curve $\boldsymbol\gamma$. Let $\boldsymbol\gamma$ be expressable as a word in only the $c_k$ so that, viewed in the plane $\pplane$, $\boldsymbol\gamma$ is a graph.  Let $\lfloor\boldsymbol\gamma\rfloor$ be the curve consisting of linear segments that, at each integer in the horizontal direction, intersects the lattice point immediately below $\boldsymbol\gamma$. The curve $\lceil\boldsymbol\gamma\rceil$ is defined similarly, by instead  pushing up to the lattice points immediately above $\boldsymbol\gamma$; see Figure \ref{fig:floor-and-ceiling}. Note that this is closely related to the {\em pegboard diagrams} introduced  in \cite{HRW}.  %Following this notation, since we will be merging with a solid torus and the invariant $\HFhat(D^2\times S^1)$ is always a line of some slope $\frac{r}{s}$ in the plane $\pplane$ \cite{HRW}, we will use $\lfloor\frac{r}{s}\rfloor$ to denote this slope pushed down to the nearest lattice points. Similarly, but with less potential for confusion with the floor of a rational number, given a slope $\alpha$ (for instance, $\alpha = r\ell + sm$ in the conventions of the next section), $\lfloor\alpha\rfloor$ denotes the slope in $\pplane$ pushed down to the nearest lattice points.

Recall that the immersed multicurve $\HFhat(M)$ coming from a bordered 3-manifold $M$ lives in the cylinder $(\R^2 \setminus \Z^2 )/\langle \lambda \rangle$, where $\lambda$ corresponds to the homological longitude of $M$. Equivalently, we think of this as a multicurve in $\R^2\setminus\Z^2$ which is invariant under the action of $\lambda$. We will say that such a curve has \emph{horizontal period $p$} if translation by $\lambda$ moves $p$ units in the horizontal direction.

\begin{corollary}\label{crl:recipe}
Let $\boldsymbol{\vartheta}$ be an immersed multicurve with local systems in $\pplane$ with horizontal period $p$, and let $\boldsymbol\gamma$ be a curve $\pplane$ with no vertical tangent lines (i.e. the graph of a function) with horizontal period $q$, with $p$ and $q$ relatively prime. Then the immersed multicurve for $\me(\boldsymbol\gamma,\boldsymbol\vartheta)$ is obtained by adding $\lfloor \boldsymbol{\gamma} \rfloor$ to $\boldsymbol{\vartheta}$ vertically. That is we find the image $\boldsymbol{\vartheta}$ under the transformation of $\pplane$ which translates along each vertical line to take the horizontal axis to $\lfloor \boldsymbol{\gamma} \rfloor$.
\end{corollary}

\begin{proof}
This is the main thrust of Figure \ref{fig:merge}: For a component homologous to $\lambda$, the new cyclic word moves $q$ columns to the right in the grid each time it traverses the grid vertically. Since $p$ and $q$ are relatively prime the new word makes $p$ vertical passes, tracing out the entire grid, before returning to the starting point. The new word is $p$ copies of the word representing $\boldsymbol\vartheta$, with the indices on type $c$ letters shifted according to the column in the grid; this corresponds to $p$ copies of the fundamental region in $\boldsymbol{\tilde\vartheta}$, each of which moves $q$ units to the right, with a plane shear applied along each column of lattice points. The magnitude of each shear is determined by the index of the corresponding letter in $\boldsymbol\gamma$, which amounts to shifting each column upwards by the height of $\lfloor \boldsymbol{\tilde\gamma} \rfloor$ in that column. The resulting curve has horizontal period $pq$. For a nullhomologous component, the grid gives rise to $p$ separate cyclic words, each traversing the grid vertically once starting in a different column. Each word is a copy of $\boldsymbol\vartheta$ with shifted indices on type $c$ letters. The nullhomologous component of $\boldsymbol\vartheta$ lifts to infinitely many copies of the same closed curve in $\boldsymbol{\tilde\vartheta}$, which are translations of each other by multiples of $\lambda$. Taking $p$ consecutive copies corresponds to the $p$ cyclic words in the grid, and adding $\lfloor \boldsymbol{\tilde\gamma} \rfloor$ corresponds to the required shifts in indices.
\end{proof}

As mentioned previously, the setup of Corollary \ref{crl:recipe} is more general than we need for cabling; we will only need the case that $\boldsymbol\gamma$ is in fact a straight line of some rational slope. Note that when $\boldsymbol\gamma$ is a line of slope 1, the transformation taking the horizontal axis to $\lfloor \curveset \rfloor$ is a lift to $\pplane$ of a Dehn twist in $\ptorus$. This is a linear transformation of the plane, which we refer to as a \emph{plane shear} in the vertical direction. The case that $\boldsymbol\gamma$ is a line of rational slope is a mild generalization of this, which we call a \emph{fractional plane shear}.

%Then we make the following adjustment to deal with non-trivial local systems: Put all of the interesting switching on some $a_j$ and then observe that $\me(a_j,c_k)=a_{j}$. by tracing through the proof in \cite{HW} (see also \cite[Section 5]{HRW}).\footnote{This is where I really don't feel like explaining more.}  

\section{The proof of Theorem \ref{thm:main-cable}}\label{sec:proof}% !TEX root = ../cabling.tex

% \section{Proof of Main Theorem}
% \label{sec:proof}

In order to complete the proof of our theorem, we need to connect the operation described in Corollary  \ref{crl:recipe} to the specific context of cabling. To do this we first set our conventions. 

\subsection{Cabling conventions}
Recall that, fixing a knot $K$, we let $M$ denote the complement $S^3 \setminus \nu(K)$ and $M_{p,q}$ denote the complement of the cable $C_{p,q}(K)$. Let $\sP$ denote a two-sphere with three disks removed (so that $\sP$ is homeomorphic to a pair of pants). The manifold $M_{p,q}$ can be obtained by gluing $M$ into one boundary component of $\sP \times S^1$ and an appropriately framed solid torus $D^2 \times S^1$ into another boundary component. We will briefly review this construction, paying particular attention to framing conventions.

Each torus boundary in this construction has a natural choice of parametrizing curves. For $\partial M$, we use a meridian $\mu$ and the Seifert longitude $\lambda$, fixing orientations on these curves with the convention that $\lambda \cdot \mu = +1$. For $\partial(D^2 \times S^1)$, we let $m$ be a meridian $\partial D^2 \times \{\pt\}$ and let $\ell$ be the longitude $\{\pt\}\times S^1$, with the orientation convention that $m\cdot\ell = +1$. For $i \in \{1,2,3\}$, the $i$th boundary component of the $S^1$-bundle $\sP \times S^1$ is parametrized by a fiber $f_i = \{\pt\}\times S^1$ and $b_i = \partial_i(\sP)\times\{\pt\}$, where $\partial_i(\sP)$ denotes the $i$th boundary component of the base surface $\sP$. We set orientations on these curves so that $b_i\cdot f_i = +1$.

The third boundary component of $\sP \times S^1$ will ultimately become the boundary of $M_{p,q}$, however it is helpful for the moment to fill this third boundary component in with a solid torus in a trivial way so that $\sP \times S^1$ becomes $A \times S^1$, where $A$ is an annulus. This solid torus can be removed later by deleting a neighborhood of a fiber of $A \times S^1$. We glue $D^2 \times S^1$ to the first boundary component of $\sP \times S^1$ (now $A \times S^1$) such that $f_1$ is identified with $p\ell + q m$; this means that $b_1$ is identified with $r\ell + s m$ for some integers $r,s$ with $ps-qr = -1$ (we can choose $r$ and $s$ arbitrarily subject to this condition, but the choice effects the framings on the resulting boundary components). The result of this gluing is a solid torus, equipped with a Seifert fibration in which the core of the solid torus is a singular fiber and the regular fibers wind $p$ times longitudinally and $q$ times meridionally. This solid torus is glued to the knot complement $M$ such that the result is $S^3$ and the core of the solid torus is identified with $K$. As a result, a regular fiber of $D^2 \times S^1 \cup A \times S^1$ is the cable $C_{p,q}(K)$, and removing a neighborhood of one of these (or, equivalently, not filling in the third boundary of $\sP\times S^1$), yields the complement $M_{p,q}$. 

Note that inserting $A \times S^1 \cong T^2 \times [0,1]$ between $M$ and $D^2\times S^1$ amounts to a change of framing and, in particular, $f_1$ can be identified with $f_2$ and $b_1$ can be identified with $-b_2$. To recover $S^3$, we want $\mu$ to be identified with $m$ and $\lambda$ to be identified with $\ell$. It follows that $f_2$ glues to $p\lambda + q\mu$ and $b_2$ glues to $-r\lambda - s\mu$. To summarize, we have
$$ \left[ \begin{matrix} f_1 \\ b_1 \end{matrix} \right] = \left[ \begin{matrix} p & q \\ r & s \end{matrix} \right] \left[ \begin{matrix} \ell \\ m \end{matrix} \right] \qquad
\left[ \begin{matrix} f_2 \\ b_2 \end{matrix} \right] = \left[ \begin{matrix} p & q \\ -r & -s \end{matrix} \right] \left[ \begin{matrix} \lambda \\ \mu \end{matrix} \right]$$
with $ps-qr = -1$. Inverting these matrices:
$$\left[ \begin{matrix} \ell \\ m \end{matrix} \right] = \left[ \begin{matrix} -s & q \\ r & -p \end{matrix} \right] \left[ \begin{matrix} f_1 \\ b_1 \end{matrix} \right] \qquad
\left[ \begin{matrix} \lambda \\ \mu \end{matrix} \right] = \left[ \begin{matrix} -s & -q \\ r & p \end{matrix} \right] \left[ \begin{matrix} f_2 \\ b_2 \end{matrix} \right]$$
If we do not fill in the third boundary of $\sP\times S^1$ in the construction above the resulting manifold with torus boundary is $M_{p,q}$; it is clear that the meridian $\mu_C$ of the cable knot is given by $b_3$. While not required, as will follow from the computation below, one can check that the Seifert longitude $\lambda_C$ of the cable knot is given by $-f_3 + q^2 b_3$. 

%\subsection{Computing $\HFhat(M_{p,q})$}.
\subsection{Applying the merge operation}
We are interested in obtaining the immersed curve-set $\HFhat(M_{p,q})$ from the immersed curve-set $\HFhat(M)$. We can do this by applying the merge operation to $\HFhat(M)$ and $\HFhat(D^2 \times S^1)$, keeping in mind the framings discussed above. Following Section \ref{sec:curves} and the conventions in \cite{HW}, the first step is to draw lifts of both curve-sets in the plane with respect to the parametrization by $f_i$ and $b_i$ (or, more precisely, by the curves in $\partial(D^2\times S^1)$ or $\partial M$ which are identified with $f_i$ and $b_i$), such that $b_i$ is the horizontal direction and $f_i$ is the vertical direction. Recall our convention is that $b_i \cdot f_i = 1$, but we are now considering the plane as a lift of $\partial(D^2\times S^1)$ or of $\partial M$, which are identified with boundary tori of $\sP\times S^1$ by an orientation reversing diffeomorphism, so if we take $b_i$ to be the positive horizontal axis then $f_i$ is the negative vertical axis. Note that we could instead choose the opposite orientation for both axes, but this ambiguity can be ignored since immersed curves for bordered invariants are symmetric under the elliptic involution of the torus by \cite[Theorem 7]{HRW-companion}. Since $\HFhat(D^2 \times S^1)$ is the meridian $m = rf_1 - pb_1$, this curve is simple to describe in the relevant basis: it is a line of slope $\tfrac r p$ \cite{HRW}. The second step is to apply Corollary \ref{crl:recipe} by taking the vertical sum of $\HFhat(M)$ and $\lfloor m \rfloor$; note that the conditions of the Corollary are satisfied because components $\HFhat(M)$ are homologous to zero or to the rational longitude $\lambda$, which moves horizontally by $q$ units, while $m$ moves horizontally be $p$ units. The result is (a lift to the plane of) $\HFhat(M_{p,q})$, though given with respect to the framing $(f_3, b_3)$ rather than the usual $(\mu_C, \lambda_C)$; see Figure \ref{fig:proof-merge-framing}.

\begin{figure}
\labellist 
\pinlabel {\color{darkBlue} $b$} at 100 271
\pinlabel {\color{lightBlue} $f$} at 75 245

\pinlabel {\color{lightPeach} $\mu$} at 133 245
\pinlabel {\color{darkPeach} $\lambda$} at 60 271

\pinlabel {\color{lightGreen} $m$} at 28 245

\pinlabel {\color{lightGreen} $m$} at 166 130
\pinlabel {\color{darkGreen} $\floor{m}$} at 165 110
\pinlabel {\color{darkBlue} $b$} at 165 75
\pinlabel {\color{lightBrown} $\ceiling{\mu}$} at 165 35

\pinlabel {\color{darkPeach} $\lambda = -sf -qb$} at 75 205
\pinlabel {\color{lightPeach} $\mu = \phantom{-}rf + pb$} at 75 190
\pinlabel {\color{lightGreen} $m = \phantom{-}rf -pb$} at 73 175

\endlabellist
\includegraphics[scale=.95]{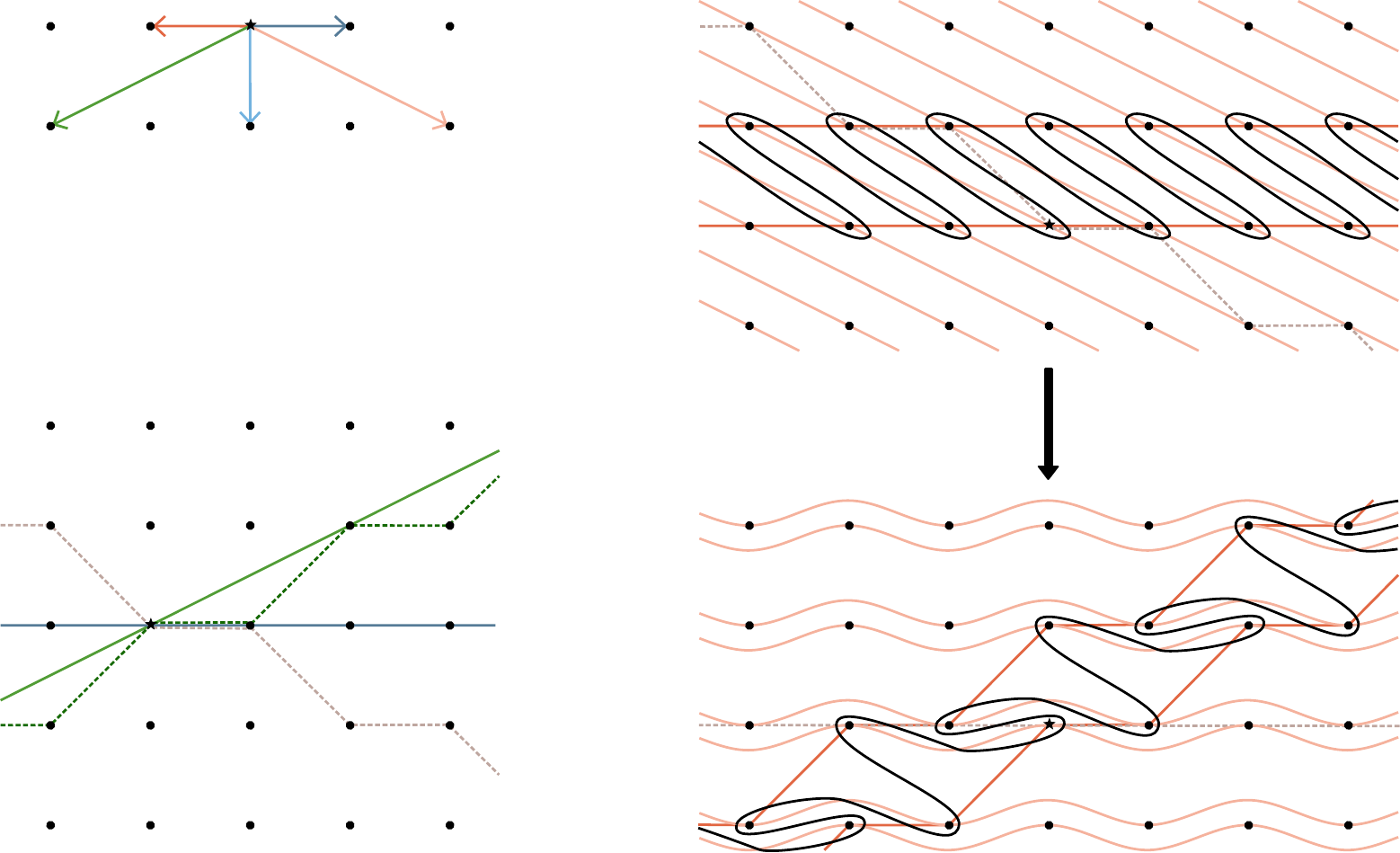}
\caption{The fractional plane shear in the vertical direction associated with computing a $(p,q)$-cable, viewed with respect to the $(-f, b)$ framing. Other relevant curves are shown, with respect to this framing, in the top left. The bottom left shows a copy of $m$ through the origin and the corresponding curve $\floor{m}$ obtained by dropping down to the highest peg below $m$ in each column. This curve serves as a ``key" for the plane shear--that is, we we shift each column of pegs upward by the height of $\floor{m}$ in that column. Thus the plane shear is determined by the fact that it takes $b$ to $\floor{m}$, or equivalently that it takes $\ceiling{\mu}$ to $b$. The right shows the effect of this shear on the curve for the right handed trefoil. For the concrete example in the figure, $(p,q) = (2,1)$ and $(r,s) = (1,0)$.}\label{fig:proof-merge-framing}
\end{figure}
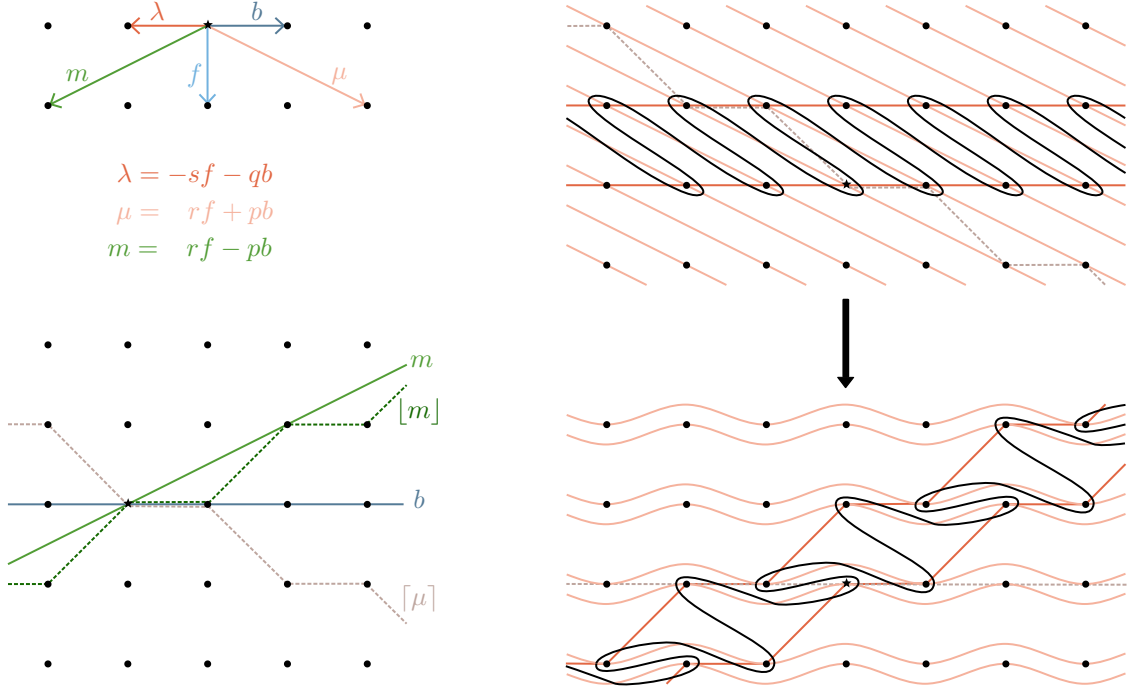

%\emph{A change of perspective:}
While the previous paragraph gives a complete procedure for computing $\HFhat(M_{p,q})$, performing the change of basis to draw the curve-set $\HFhat(M)$ with respect to the $(f, b)$ framing can be cumbersome. Instead, we can follow the same operation but view the plane with respect to $(\mu, \lambda)$, the preferred framing for $\partial M$, throughout the process. Now, instead of shifting pegs in each vertical column, we shift along lines parallel to the fiber direction; since $f_2 = p\lambda + q\mu$, this is a line of slope $\tfrac q p$. To keep track of how much to shift along each line of slope $\tfrac q p$ we can draw a copy of the piecewise linear curve $\ceiling{\mu}$; note that this is obtained from a vertical line $\mu$ through the origin by pushing each point $(0, \tfrac n p)$ on $\mu$ leftward along a line of slope $\tfrac q p$ to the first lattice point it encounters (see Figure \ref{fig:proof-knot-framing}, top left). To perform the cable operation, we shear along lines of slope $\tfrac q p$ to bring this curve $\ceiling{\mu}$ to $b$ (see Figure \ref{fig:proof-knot-framing}, bottom left). 

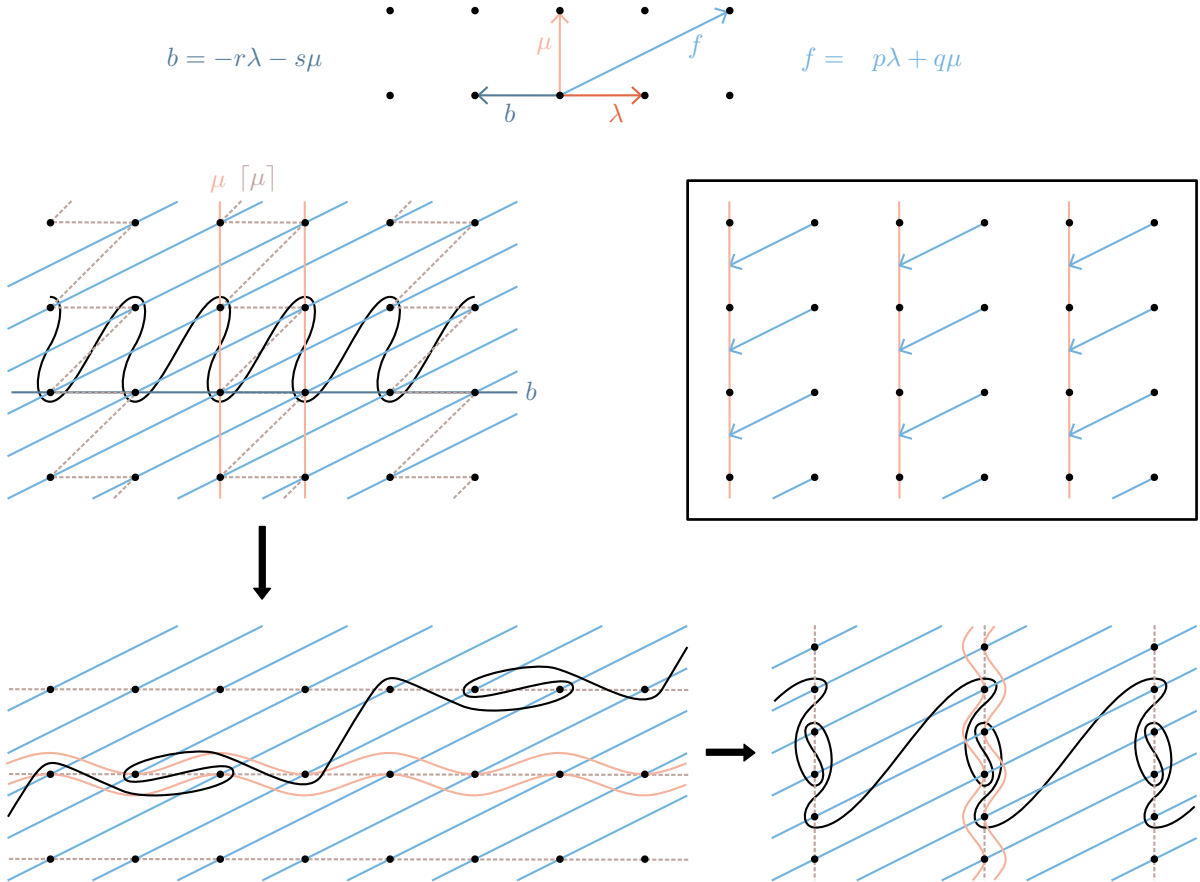
\begin{figure}[t]
\labellist 
\pinlabel {\color{darkBlue} $b$} at 190 290
\pinlabel {\color{lightBlue} $f$} at 260 315

\pinlabel {\color{lightPeach} $\mu$} at 203 315
\pinlabel {\color{darkPeach} $\lambda$} at 230 290

\pinlabel {\color{darkBlue} $b$} at  198 185

\pinlabel {\color{lightBrown} $\ceiling{\mu}$} at 95 265
\pinlabel {\color{lightPeach} $\mu$} at 80 263

\pinlabel {\color{darkBlue} $b = -r\lambda -s\mu$} at 90 310
\pinlabel {\color{lightBlue} $f = \phantom{-}p\lambda + q\mu$} at 330 310

\endlabellist
\includegraphics[scale=1]{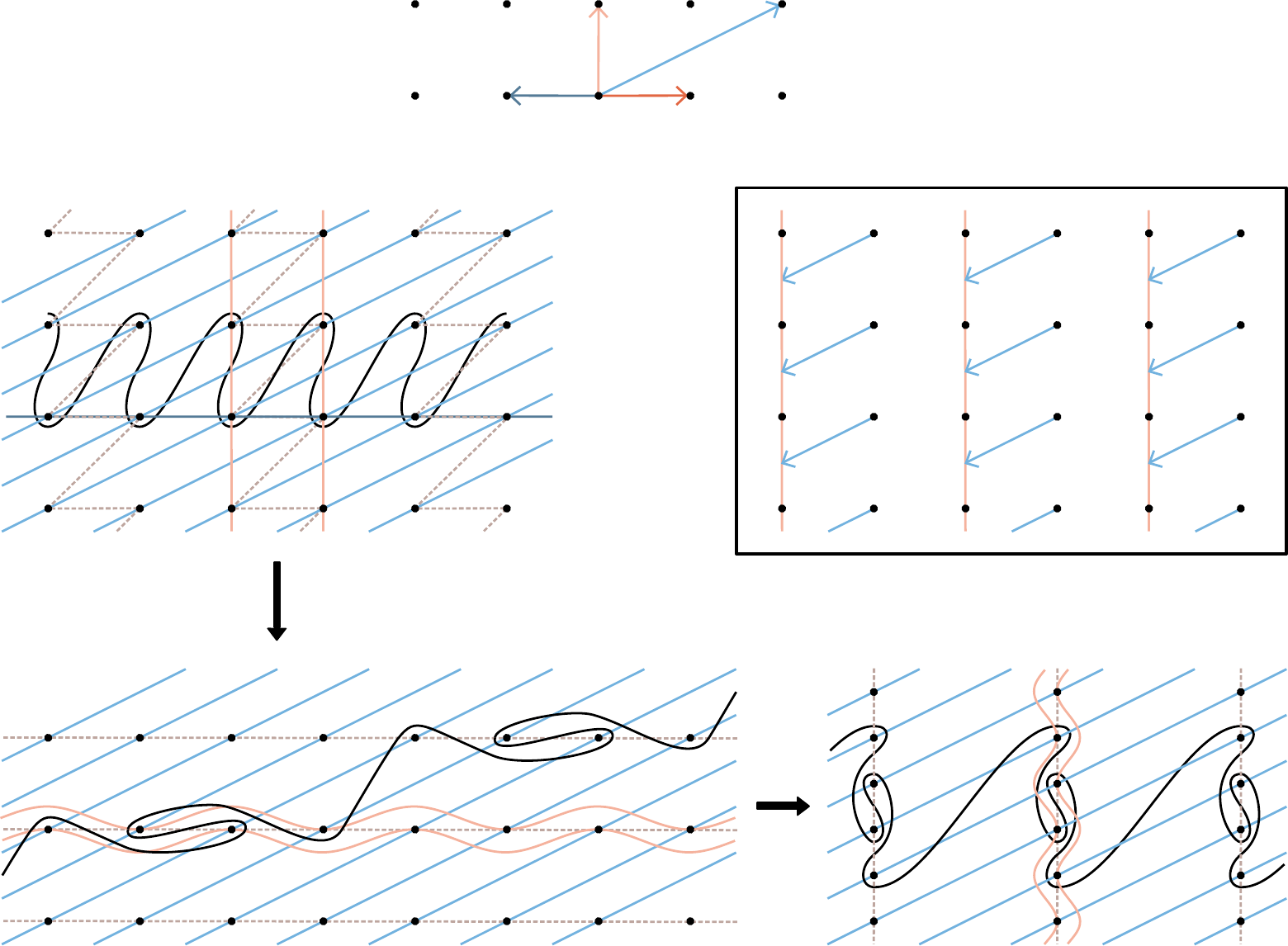}
\caption{Starting with the curve $\HFhat(M)$ drawn in the plane with respect to the standard $(\mu, \lambda)$ framing, the fractional plane shear in the $f$ direction which takes $\ceiling{\mu}$ to $b$ produces the curve $\HFhat(M_{p,q})$, though not in terms of a convenient parametrization. shearing back partially along $f$ gives the curve with the standard parametrization, up to rescaling the lattice. These two steps can be combined into one, as shown in the box: each lattice point is translated leftward along lines of slope $\tfrac qp$ until its $x$ coordinate is a multiple of $p$. Note that $p$ copies of the curve $\HFhat(M)$ are involved in each copy of $\HFhat(M_{p,q})$. The figure shows the case of of the $(2,1)$ cable of the right handed trefoil.}\label{fig:proof-knot-framing}
\end{figure}

We can now start with the curve $\HFhat(M)$ represented in terms of its standard framing $(\mu, \lambda)$ and produce the immersed curve $\HFhat(M_{p,q})$ in one simple step. However, as before, the output is not given with respect to the standard framing by $(\mu_C, \lambda_C)$. Of course, it is straightforward to determine the slopes of $\mu_C$ and $\lambda_C$ in the output picture and then we simply need to change basis applying a linear map to the plane which takes these to the vertical and horizontal directions, respectively. This can always be accomplished by a sequence of (integral) plane shears in the horizontal and vertical directions. However, this too is cumbersome, so we will describe a shortcut to this reparametrization making use of a linear transformation of the plane which does not preserve the lattice. More precisely, consider the linear transformation which fixes $f$ and takes $b$ to $\mu$; this can be understood as translating each lattice point on $b$ along a line of slope $\tfrac  q p$ until it reaches the vertical line $\mu$ (see Figure \ref{fig:proof-knot-framing}, bottom right). Note that the lattice $\Z^2$ is not mapped to itself under this transformation, but rather its image is $p\Z \times \tfrac 1 p \Z$. Even so, in this new deformed lattice the directions corresponding to $\mu_C$ and $\lambda_C$ are vertical and horizontal, as desired, and we can recover the usual lattice by ending with another linear transformation which scales and compresses by a factor of $p$ in the vertical direction and horizontal direction, respectively.

Finally, we mention that there are now two steps which involve shearing along the lines of slope $\tfrac q p$: the fractional plane shear taking $\ceiling{\mu}$ to $b$ (this transformation is not linear), and the linear transformation taking $b$ to $\mu$. These steps can be combined in one by shearing along lines of slope $\tfrac q p$ to push $\ceiling{\mu}$ onto $\mu$. In other words, every $p$th vertical column of lattice points is fixed, while all other points are pushed leftward along lines of slope $\tfrac q p$ until they reach a vertical line containing one of the fixed columns (see boxed portion of Figure \ref{fig:proof-knot-framing}). This proves Theorem \ref{thm:main-cable}.

%\subsection*{Acknowledgements} The authors thank Jake Rasmussen for many enjoyable discussions associated with this work. 

\bibliographystyle{plain}
\bibliography{references/bibliography}

\end{document}